\documentclass[reqno]{amsart}
\pdfoutput 1

\usepackage{lineno}

\usepackage[ruled,vlined,norelsize,noend]{algorithm2e}

\usepackage{tikz-cd}
\usetikzlibrary{calc}

\usepackage[all,hyperref,numberbysection,mathbb]{bi-discrete}

\usepackage[backend=biber,maxbibnames=5,maxalphanames=5,style=alphabetic,bibencoding=utf8,giveninits,url=false,isbn=false,doi=true]{biblatex}
\renewbibmacro{in:}{}
\AtBeginBibliography{\small}

\usepackage{listings}
\usepackage{color} 

\definecolor{codegreen}{rgb}{0,0.6,0}
\definecolor{codegray}{rgb}{0.5,0.5,0.5}
\definecolor{codepurple}{rgb}{0.58,0,0.82}
\definecolor{backcolour}{rgb}{0.95,0.95,0.92}

\lstdefinestyle{mystyle}{
    backgroundcolor=\color{backcolour},   
    commentstyle=\color{codegreen},
    keywordstyle=\color{magenta},
    numberstyle=\tiny\color{codegray},
    stringstyle=\color{codepurple},
    basicstyle=\ttfamily\footnotesize,
    breakatwhitespace=false,         
    breaklines=true,                 
    captionpos=b,                    
    keepspaces=true,                 
    numbers=left,                    
    numbersep=5pt,                  
    showspaces=false,                
    showstringspaces=false,
    showtabs=false,                  
    tabsize=2
}

\usepackage{tikz}
\usepackage{pgfplots}
\pgfplotsset{
  width=.65\linewidth,
  axis background/.style={fill=black!5!white},
  grid style={densely dotted,semithick},
  legend style={
    legend columns=1,
    legend pos=outer north east
  },
  compat=newest 
}
\pgfplotscreateplotcyclelist{MyColors}{%
    {red,mark = *,every mark/.append style={solid,scale=0.4,fill=red}},
    {teal,mark = x,every mark/.append style={solid,scale=0.4,fill=teal}},
    {blue,mark = square*,every mark/.append style={solid,scale=0.4,fill=blue}},
{dotted,red,mark = *,every mark/.append style={solid,scale=0.4,fill=red}},
    {dotted,teal,mark = x,every mark/.append style={solid,scale=0.4,fill=teal}},
    {dotted,blue,mark = square*,every mark/.append style={solid,scale=0.4,fill=blue}},
{dash dot,red,mark = *,every mark/.append style={solid,scale=0.4,fill=red}},
    {dash dot,teal,mark = x,every mark/.append style={solid,scale=0.4,fill=teal}},
    {dash dot,blue,mark = square*,every mark/.append style={solid,scale=0.4,fill=blue}},}

\usepackage[font=small]{caption}

\providecommand{\dy}{\, \mathrm{d}y}
\providecommand{\dx}{\, \mathrm{d}x}
\providecommand{\ds}{\, \mathrm{d}s}
\providecommand{\dt}{\, \mathrm{d}t}
\providecommand{\dz}{\, \mathrm{d}z}
\providecommand{\dr}{\, \mathrm{d}r}
\providecommand{\tria}{\mathcal{T}}

\providecommand{\midpoint}{\textup{mid}}

\newcommand{\Vander}{\mathrm{V}}
\newcommand{\edge}{f}
\newcommand{\edges}{\mathcal{F}}
\newcommand{\nodes}{\mathcal{N}}
\newcommand{\bfcurl}{\operatorname{\mathbf{curl}}}

\newcommand{\Matb}{\mathtt{b}}
\newcommand{\Matc}{\mathtt{c}}

\newcommand{\Zspace}{Z}  
\newcommand{\Zsing}{Z_s} 
\newcommand{\Zred}{Z_r}  
\newcommand{\Vspace}{V}  
\newcommand{\Qspace}{Q}  
\newcommand{\Vred}{V_r}  
\newcommand{\Qred}{Q_r}  

\newcommand{\Vspacediv}{V_{\divergence}} 
 
\providecommand{\Zienkiewicz}{Zienkiewicz\xspace} 
\providecommand{\Matlab}{MATLAB\xspace}

\begin{document}

\author[L.\ Diening]{Lars Diening}
\author[J.\ Storn]{Johannes Storn}
\author[T.\ Tscherpel]{Tabea Tscherpel}

\address[L.\ Diening]{Department of Mathematics, Bielefeld University, Postfach 10 01 31, 33501 Bielefeld, Germany}
\email{lars.diening@uni-bielefeld.de}
\address[J.\ Storn]{Faculty of Mathematics \& Computer Science, Institute of Mathematics, Leipzig University, Augustusplatz 10, 04109 Leipzig, Germany}
\email{johannes.storn@uni-leipzig.de}
\address[T.\ Tscherpel]{Department of Mathematics, Technische Universit\"at Darmstadt, Dolivostra\ss e 15, 64293 Darmstadt, Germany}
\email{tscherpel@mathematik.tu-darmstadt.de}



\title[Exact integration for rational finite elements]{Exact Integration for singular Zienkiewicz and Guzm\'an--Neilan Finite Elements with Implementation}

\begin{abstract} 
We develop a recursive integration formula for a class of rational polynomials in 2D. 
Based on this, we present implementations of finite elements that have rational basis functions. 
Specifically, we provide simple \Matlab implementations of the singular \Zienkiewicz and the lowest-order \Guzman{}--Neilan finite element in 2D. 
\end{abstract}

\maketitle


\section{Introduction}

\noindent The majority of finite elements rely on polynomial functions because they are simple to implement by use of exact integration and differentiation. 
However, achieving certain types of conformity, such as $H^2$-conformity or exact divergence constraints, can be challenging with polynomials; 
ensuring these properties typically requires a large number of degrees of freedom when working with polynomial functions. 
By using rational functions these constraints can be satisfied with considerably fewer degrees of freedom. 
However, to the best of our knowledge, no exact integration formulas for rational polynomials on triangles have been established in the literature. 
As a result, finite elements based on rational basis functions have not been  studied extensively. 
In this work we develop an exact integration in 2D for a class of rational polynomials and we apply it to several rational finite elements. 
This allows for a straightforward exact implementation, and paves  the way for further investigations of rational finite elements. 

\subsubsection*{Approximation with rational functions} 
For certain problems rational functions offer better approximation properties than polynomials, see, e.g.,~\cite{Newman1964,DeVoreYu1986}. 
Indeed, for functions with certain singularities this is evident. 
Designing algorithms for rational approximations is a recent topic of research~\cite{AustinKrishnamoorthyLeyfferEtAl2021,TrefethenNakatsukasaWeideman2021,BoulleHerremansHuybrechs2024}. 
Compared to the polynomial case many questions are still open. 
Indeed, exact quadrature formulas have been developed mostly for 1D problems, e.g.,~\cite{Gautschi1993, Gautschi2001}, for 2D circular domains without radial dependence~\cite{BultheelGonzalezVeraHendriksenEtAl2001, DeckersBultheelPerdomoPio2011} and for tensorial problems.
A prominent example for the approximation with rational functions are the so-called NURBS used in isogeometric analysis ~\cite{HughesCottrellBazilevs2005}, see also~\cite{BuffaGantnerGiannelliEtAl2022}. 
Aside from few exceptions~\cite{ManniSpeleers2016}, they are based on a tensorial structure.  
For the quadrature usually a Duffy transform~\cite{SauterSchwab1997} is employed to resolve the singularities, and then quadrature formulas in 1D are applied in the tensorial form. 
This approach leads to inexact quadrature, but it offers exponential convergence in the number of quadrature points, see, e.g.,~\cite{Auricchio2012,DoelzHarbrechtKurzEtAl2020}. 
For complex geometries also triangular meshes and the use of rational functions on them are of interest. 
For those, in~\cite{BoulleHerremansHuybrechs2024} the problem of approximation has been addressed by splitting triangles into three quadrilaterals and by transforming them to the unit cube. 
However, for multivariate rational functions without tensor product structure, exact numerical integration seems to be an open problem, even in the case of dimension $d=2$.

\subsubsection*{Finite element methods with rational functions}
Finite element spaces based on low order polynomials fail to satisfy certain conformity properties. 
For example, the \Zienkiewicz element \cite[Fig.~2.2.16]{Ciarlet2002} is not $H^2$-conforming, and for the Taylor--Hood element \cite[Sec.~8.8]{Boffi2013} the corresponding space of discretely divergence-free velocity functions is not exactly divergence-free. 
However, in some situations such structure preservation in the form of conformity is crucial. 
Indeed, the lack of $H^2$-conformity can cause difficulties in the approximation of solutions to the biharmonic equation; see~\cite{IronsRazzaque72,Shi84} for the failure of convergence of the \Zienkiewicz element on certain meshes.  
To obtain $H^2$-conformity with polynomial basis functions requires a high number of degrees of freedom. 
For example, both the Argyris element~\cite{ArgyrisFriedScharpf1968}, see also~\cite[Ch.~2.2.2]{Ciarlet2002} and for a hierarchical version  \cite{CarstensenHu21}, and the Morgan--Scott ``element''~\cite{MorganScott1975} is $C^1$-conforming, but its definition depends on certain mesh conditions. Both elements use polynomials of 5th order, and hence have $21$ degrees of freedom per triangle.
On the other hand, the $H^2$ conforming singular \Zienkiewicz element~\cite[Sec.~10.10]{Zienkiewicz1971}, see also~\cite[Ch.~6.6.1]{Ciarlet2002} and Section~\ref{sec:Zienkiewicz} below, has only $12$ degrees of freedom per triangle due to the use of rational basis functions. 

Divergence-free subspaces of $H^1$ and $H^2$-conforming finite element spaces are related by a de~Rham complex.  More precisely, the curl of $H^2$-conforming elements yields exactly divergence-free $H^1$-conforming elements. 
Therefore, mixed methods for incompressible fluid equations with exact divergence constraint inherit some properties from the $H^2$-conforming spaces. 
The importance of exact divergence constraints lies in the fact that they ensure pressure robustness of mixed methods for incompressible fluid equations, see~\cite{Linke09,GalvinLinkeRebholzWilson12} for specific examples and~\cite{John17} for a discussion illustrating this point in the context of computational fluid dynamics.  
There are alternatives to ensure pressure robustness by non-conforming methods such as~\cite{LedererLinkeMerdonSchoeberl17,KreuzerVerfuerthZanotti21}. 
However, they inflict extra challenges for time-dependent and for non-linear equations, some of which have not been resolved to date. 
As for the $H^2$-conforming spaces, exactly divergence-free finite elements are available for high polynomial degree, e.g., the Falk--Neilan element~\cite{FalkNeilan2013} and the Scott--Vogelius element~\cite{ScottVogelius85,GuzmanNeilan2014}. 
For the latter, using split meshes the polynomial degree can be lowered~\cite{GuzmanNeilan2018,FabienGuzmanNeilanEtAl2022}, but the number of degrees of freedom increases; see~\cite{ScottTscherpel2024} for a comparison. 
The \Guzman{}--Neilan element~\cite{GuzmanNeilan2014}, see also Section~\ref{sec:GNelement} below, uses rational basis functions for the velocity space, and requires considerably fewer degrees of freedom. 
Note that also a 3D version is available in~\cite{GuzmanNeilan2014-3D}. 

\subsubsection*{Quadrature for rational functions}

While exact quadrature formulas can be applied for rational functions with tensor product structure, the theory is much less developed for more general rational functions.
On simplices, to the best of our knowledge, no exact quadrature for rational functions is available in the literature, not even in the lowest-dimensional case $d = 2$.
In practice, the Duffy transform in combination with a quadrature of sufficiently high exactness degree is used. This approach is common in the implementation of NURBS, and was also applied to the 2D \Guzman{}--Neilan finite element in~\cite{Schneier2015}. 
Although the integration is not exact, geometric convergence in the number of quadrature points is available. 
For applications in isogeometric analysis this may be sufficient. 
However, for finite element spaces with exact divergence constraints the lack of an exact quadrature comprises the conformity, see Section~\ref{subsec:ExpPressRob}, and thereby diminishes the benefit of the method. 

\subsubsection*{Main contribution}
In this work we present an exact quadrature for a class of 2D rational functions on triangles, based on a recursion formula. 
The quadrature can be computed during an offline phase and  tabulated for efficient use. 
This opens the door for the further investigation of the singular \Zienkiewicz element and the \Guzman{}--Neilan element, and their use in practical applications. 
Our objective is to demonstrate the overall procedure, and provide an easily accessible implementation in \Matlab. Specifically, in the spirit of~\cite{AlbertyCarstensenFunken99,AlbertyCarstensenFunkenKlose02,BahriawatiCarstensen05,CarstensenGallistlHu14,Bartels15} we provide an implementation of the biharmonic problem discretized by singular \Zienkiewicz element and of the Stokes problem discretized by the lowest-order \Guzman{}--Neilan element. 
All code is available at~\cite{Code}. 
Since the focus of this work is not on highly optimized performance, any improvement in this direction is left to future work. 

\subsubsection*{Outline} 
In Section~\ref{sec:rat-poly} we present an exact recursive integration formula for a large class of rational polynomials in $d=2$ dimensions. 
The class of rational polynomials includes basis functions occurring in singular \Zienkiewicz and in \Guzman{}--Neilan finite elements, which are introduced in Section~\ref{sec:Zienkiewicz} and Section~\ref{sec:GNelement}, respectively.
Section~\ref{sec:impl-general} provides a general guideline for implementing finite element methods with such rational basis functions. 
This is illustrated through simple and short \Matlab implementations of the biharmonic problem discretized by the singular Zienkiewicz element and the Stokes problem discretized by the lowest-order \Guzman{}--Neilan element. 
The implementations are discussed in Section~\ref{sec:implZienk} and~\ref{sec:implGN}, respectively. 
We conclude our investigation with a numerical example of the impact of inexact integration on the biharmonic eigenvalue problem in Section~\ref{subsec:BiHarm} and the Stokes equation in Section~\ref{subsec:ExpPressRob}, highlighting the need for exact integration rules.

\section{Finite element spaces}
\label{sec:FEspace}

\noindent 
In this section we introduce the singular \Zienkiewicz in Section~\ref{sec:Zienkiewicz}, and the lowest-order \Guzman{}--Neilan element in Section~\ref{sec:GNelement}. 

Throughout this work we denote by~$\mathcal{T}$ a triangulation in the sense of Ciarlet~\cite[\S~2.1, p.~38]{Ciarlet2002} of a two-dimensional, bounded, polyhedral domain $\Omega$. 
By $\nodes$ we denote the set of all vertices in~$\tria$.
A triangle~$T=[v_0,v_1,v_2] \in \mathcal{T}$ is determined by its vertices $v_0,v_1,v_2 \in \mathbb{R}^2$. 
The set of all faces/edges~$\edge$ of~$T$ is $\edges(T) = \lbrace\edge_0,\edge_1,\edge_2\rbrace$, where $\edge_j$ denotes the edge opposite of vertex $v_j$ for all $j=0,1,2$. 
The outer unit normal of $f \in \edges(T)$ is denoted by $\nu_{f}$.
 By$~\midpoint(T)$ and $\midpoint(\edge)$ we denote their respective midpoints. 
Whenever indices exceed $2$, they are understood as modulo $3$, e.g., $v_3 = v_0$.  
The space of polynomials on~$T$ of maximal degree $k\in \setN_0$ is denoted by $\mathcal{P}_k(T)$. Moreover, let $\lambda_0,\lambda_1,\lambda_2 \in \mathcal{P}_1(T)$ denote the barycentric coordinates associated to the vertices $v_0,v_1,v_2$, that is, $\lambda_j(v_k) = \delta_{j,k}$ for all $j,k=0,1,2$. We denote the vector of barycentric coordinates by $\lambda = (\lambda_0,\lambda_1,\lambda_2)^\top \in [0,1]^3$. 
We shall frequently use that these functions form the partition of unity $\lambda_0 + \lambda_1 + \lambda_2 = 1$.  
The barycentric coordinates represent a mapping from $T$ onto the reference simplex $\widehat{T} = [\hat{v}_0,\hat{v}_1,\hat{v}_2] \subset \RR^3$ with vertices 
\begin{equation}\label{eq:defvhat}
\hat{v}_0 = (1,0,0)^\top,\quad \hat{v}_1 = (0,1,0)^\top,\quad\hat{v}_2  = (0,0,1)^\top.
\end{equation}
If there is no risk of misunderstanding $\nabla$ denotes the gradient with respect to $x$. 
If a gradient is taken with respect to another variable $y$ instead, for clarity we denote the gradient by $\nabla_y$. 

\subsection{Singular \Zienkiewicz element}
\label{sec:Zienkiewicz}
\begin{figure}
    \begin{tikzpicture}[scale=1]
		\coordinate (z1) at (0,0);
		\coordinate (z2) at (2,0);
		\coordinate (z3) at (1,1.73);
		\draw (z1) circle (5pt);		
		\draw (z2) circle (5pt);
		\draw (z3) circle (5pt);
		\fill (z1) circle (2pt);
		\fill (z2) circle (2pt);
		\fill (z3) circle (2pt);
		\draw (z1) -- (z2) -- (z3) -- cycle;
		
		\coordinate (y1) at (4,0);
		\coordinate (y2) at (6,0);
		\coordinate (y3) at (5,1.73);
		\draw (y1) circle (5pt);		
		\draw (y2) circle (5pt);
		\draw (y3) circle (5pt);
		\fill (y1) circle (2pt);
		\fill (y2) circle (2pt);
		\fill (y3) circle (2pt);
		\draw (y1) -- (y2) -- (y3) -- cycle;
    \def\lengthOrthoLine{.14cm}
\coordinate (Midpoint) at ($(y1)!0.5!(y3)$);
    \draw[thick] (Midpoint) -- ($(Midpoint)!+\lengthOrthoLine!90:(y3)$) 
        coordinate (OrthoEnd);
    \draw[thick] (Midpoint) -- ($(Midpoint)!-\lengthOrthoLine!90:(y3)$);
    
	\coordinate (Midpoint2) at ($(y1)!0.5!(y2)$);
    \draw[thick] (Midpoint2) -- ($(Midpoint2)!+\lengthOrthoLine!90:(y2)$) 
        coordinate (OrthoEnd);
    \draw[thick] (Midpoint2) -- ($(Midpoint2)!-\lengthOrthoLine!90:(y2)$);

	\coordinate (Midpoint3) at ($(y3)!0.5!(y2)$);
    \draw[thick] (Midpoint3) -- ($(Midpoint3)!+\lengthOrthoLine!90:(y2)$) 
        coordinate (OrthoEnd);
    \draw[thick] (Midpoint3) -- ($(Midpoint3)!-\lengthOrthoLine!90:(y2)$);
  \end{tikzpicture}
\caption{Degrees of freedom in the \Zienkiewicz{} element and the reduced singular \Zienkiewicz (left) as well as in the singular \Zienkiewicz (right) element. 
Dots denote point evaluation, circles denote evaluation of the gradient, and straight lines denote the evaluation of the normal derivative.  
	}\label{fig:zien}
\end{figure}%

The singular \Zienkiewicz element~\cite[Thm.~6.1.4]{Ciarlet2002} is a $C^1$-conforming finite element. 
It is an enrichment of the \emph{\Zienkiewicz{} element}, which is a reduced cubic Hermite element, see~\cite[Thm.~2.2.9 and Fig.~2.2.16]{Ciarlet2002}.
For this reason, let us first recall the \Zienkiewicz{} element. 
It is globally $C^0$-conforming and gradients are continuous in the vertices. 
Its 9~degrees of freedom on each triangle~$T$ are the point values of the function and the point values of the gradient at its vertices, cf.~Figure~\ref{fig:zien}. 
The local space is a subspace of $\mathcal{P}_3(T)$, and the latter has dimension~$10$. 
Since the local space has only 9~degrees of freedom, one considers only polynomials $p \in \mathcal{P}_3(T)$ which satisfy the additional condition
\begin{equation}\label{eq:Hermite}
  \psi(p) \coloneqq 6\, p\big(\textup{mid}(T)\big)  - 2 \,\sum_{j=0}^2 p(v_j) + \sum_{k=0}^2 \nabla p(v_k) \cdot \big(v_k - \textup{mid}(T)\big) = 0.
\end{equation}
In other words, the local space reads $\Zspace(T) \coloneqq \lbrace p \in \mathcal{P}_3(T) \colon \psi(p) = 0\rbrace$. 

\begin{lemma}[\Zienkiewicz{} element]\label{lem:Zien}
  \leavevmode%
  \begin{enumerate}
  \item 
    The local space $\Zspace(T)$ can be represented as 
    \begin{equation*}
      \Zspace(T) = \mathcal{P}_2(T) \oplus  
      \linearspan \set{ \lambda^2_j \lambda_{j+1}- \lambda_{j+1}^2 \lambda_j}_{j = 0,1,2}.
    \end{equation*}
    Its degrees of freedom are the values of $p$ and $\nabla p$ at the vertices of $T$ for $p \in Z(T)$, see Figure~\ref{fig:zien} (left). 
  \item 
    The corresponding global space $\Zspace$ is $C^0$-conforming and is given by
    \begin{align*}
      \Zspace &= \lbrace v \in H^1(\Omega)\colon v|_T\in \Zspace(T)\text{ for all }T\in \tria\rbrace. 
    \end{align*}
  \end{enumerate}
\end{lemma}
\begin{proof}
  Thanks to the structure of $\psi$ based on a Taylor expansion of second order, one can see that $\psi(p) = 0$ for any $p \in \mathcal{P}_2(T)$. 
  Since $\psi$ is linear and invariant under affine transformation, we obtain for any $i,j \in \{0,1,2\}$ with $i \neq j$
  \begin{align*}
  \psi(\lambda_i^2 \lambda_j - \lambda_i \lambda_j^2)
   =   \psi(\lambda_i^2 \lambda_j) - \psi(\lambda_i \lambda_j^2) 
  =    \psi(\lambda_i^2 \lambda_j) - \psi(\lambda_j \lambda_i^2) = 0.
  \end{align*} 
  Hence, $\psi(p)=0$ holds also for all $p\in \mathcal{P}_2(T) \oplus \linearspan\set{ \lambda^2_j \lambda_{j+1}- \lambda_{j+1}^2 \lambda_j}_{j = 0,1,2}$.  Since the basis functions are linearly independent and the dimensions agree, this shows the representation of~$\Zspace(T)$.  
  The second statements can be found in~\cite[Thm.~2.2.9 and 2.2.10]{Ciarlet2002}.
\end{proof}

There exist modifications of the polynomial space $Z(T)$ with improved properties~\cite{Wang2007}. 
However, the following considerations show the difficulties when aiming for $C^1$-conformity.

Let $T_-$ and $T_+$ be two neighboring triangles with joint edge $\edge = T_-\cap T_+$ and normal vector $\nu$ of $f$. 
Let $p_- \in \Zspace(T_-)$ and $p_+ \in \Zspace(T_+)$ be local \Zienkiewicz{} polynomials that coincide at their common degrees of freedom and set 
\begin{equation}\label{eq:DefP}
  p \colon T_- \cup T_+ \to \mathbb{R}\qquad\text{with } \, p|_{T_-} \coloneqq p_- \,\text{ and } \, p|_{T_+} \coloneqq p_+.
\end{equation}
Then $p$ is continuous. 
However, in general $p$ is not continuously differentiable, since the normal derivative of $p$ at~$f$ may jump. 
To achieve $C^1$-conformity, additional degrees of freedom are needed on the joint edge~$f$, which in turn requires enriching the local spaces. 
For this purpose an extra basis functions $B_f$ is needed that allows to adjust the normal derivative at~$f$. More precisely, we require that
\begin{subequations}
  \label{eq:cond-Bf}
  \begin{align}
  \label{eq:cond-Bf1}
   &\text{$B_f=0$ on $\partial T_-$ and $\partial T_+$,}
    \\
  \label{eq:cond-Bf2}
   &\text{$\nabla B_f=0$ on $\partial (T_- \cup T_+)$,}
    \\
  \label{eq:cond-Bf3}
   &\text{$\partial_\nu B_f \in \mathcal{P}_2(f)$  and }
   \text{$\partial_\nu B_f(\midpoint{f)} \neq 0$.}
  \end{align}
\end{subequations}
This ensures that~$B_f$ does not interfere with the other degrees of freedom.

\begin{remark}[Polynomial bases]
  \label{rem:not-polynomial}
  The properties in~\eqref{eq:cond-Bf} cannot be achieved using polynomials which we will discuss here. This is the reason, why non-polynomial basis function like rational bubbles are needed for this approach, see Lemma~\ref{lem:RatBubbles}.

We consider the simplices 
\begin{equation*}
T^-=\left[
  \begin{pmatrix}
    0 \\ 0
  \end{pmatrix},
  \begin{pmatrix}
    1 \\ 0
  \end{pmatrix},
  \begin{pmatrix}
    0 \\ 1
  \end{pmatrix}
  \right]\qquad\text{and}\qquad T^+=\left[
  \begin{pmatrix}
    1 \\ 1
  \end{pmatrix},
  \begin{pmatrix}
    1 \\ 0
  \end{pmatrix},
  \begin{pmatrix}
    0 \\ 1
  \end{pmatrix}
  \right].
\end{equation*}
  Their common face reads $f = [  \begin{psmallmatrix}
    1 \\ 0
  \end{psmallmatrix},
  \begin{psmallmatrix}
    0 \\ 1
  \end{psmallmatrix}
  ]$. 
  For contradiction, suppose that there is a polynomial~$B_f$ satisfying~\eqref{eq:cond-Bf}.  
  Then, the first two conditions require  that for some polynomial~$p$ one has
  \begin{align*}
    B_f(x,y) = x^2y^2(1-x-y)p(x,y)\qquad\text{for all }(x,y) \in T^-.
  \end{align*}
  This implies that $\partial_\nu B_f(x,1-x)= -\sqrt{2} x^2(1-x)p(x,1-x)$ on~$f$. The condition $\partial_\nu B_f \in \mathcal{P}_2(f)$ can only be fulfilled if $p(x,1-x)=0$, but then we have $\partial_\nu B_f=0$ on~$f$. This contradicts~\eqref{eq:cond-Bf3}.
\end{remark}

The singular \Zienkiewicz element~\cite[Thm.~6.1.4]{Ciarlet2002} uses a rational bubble function for~$B_f$.
More precisely, for $j=0,1,2$ let  $f_j = [v_{j+1},v_{j+2}]$ denote the edge connecting the vertices $v_{j+1}$ and $v_{j+2}$ of $T$. 
 In other words, $f_j$ is the edge opposite of $v_j$ in $T$. Let $\nu_{f_j}$ denote the outer unit normal vector of $f_j$.
We denote the cubic element bubble by $b_T\coloneqq \lambda_0 \lambda_1 \lambda_2$, the quadratic edge bubble by $b_{\edge_j} \coloneqq \lambda_{j+1} \lambda_{j+2}$ and the rational bubble by
\begin{equation}\label{eq:DefRatBubble}
  B_{\edge_j} \coloneqq  \frac{b_T b_{\edge_j}}{(\lambda_j + \lambda_{j+1})(\lambda_j + \lambda_{j+2})} = \frac{\lambda_0 \lambda_1 \lambda_2 \lambda_{j+1} \lambda_{j+2}}{(1-\lambda_{j+1})(1-\lambda_{j+2})}.
\end{equation}
The quintic polynomial in the numerator ensures~\eqref{eq:cond-Bf1} and~\eqref{eq:cond-Bf2}. 
On~$f_j$ the denominator reduces to $\lambda_{j+1}\lambda_{j+2}$, so on~$f_j$ the function $B_{f_j}$ is a cubic polynomial. 
This ensures that its derivative is quadratic on~$f_j$. 
The following lemma summarizes the important properties of~$B_{f_j}$.

\begin{lemma}[{Rational bubbles, e.g.,~\cite[Lem.~2.1]{GuzmanNeilan2014}}]\label{lem:RatBubbles}
  For all $T\in \tria$ and all indices $j = 0,1,2$ with edges $\edge_j$ of $T$,  the rational bubble $B_{\edge_j}$ is in $C^1(T)$ and satisfies
  \begin{enumerate}
  \item $B_{\edge_j}|_{\partial T} = 0$,
  \item $\nabla B_{\edge_j}|_{\edge_k} = 0$ \quad for all $k=0,1,2$ with $k\neq j$,
  \item $\nabla B_{\edge_j}|_{\edge_j}  = (\nabla b_T|_{f_j} \cdot \nu_{\edge_j}) \nu_{\edge_j}  = - b_{f_j} \nabla \lambda_j  = - \abs{\nabla \lambda_j} b_{f_j} \nu_{\edge_j} \in \mathcal{P}_2(\edge_j)^2$.
    \label{itm:SameDerivAsBubble}
  \end{enumerate}
\end{lemma}

The local space of the singular \Zienkiewicz element is defined as
\begin{equation}\label{eq:DefZienkiwiecz}
	\Zsing(T) \coloneqq \Zspace(T) \oplus \linearspan\lbrace B_{\edge_j}\colon j=0,1,2\rbrace.
\end{equation}
\begin{lemma}[Singular \Zienkiewicz element {\cite[Thm.~6.1.4]{Ciarlet2002}}]\label{lem:sing-Zien} \hfill
  \begin{enumerate}
  \item 
    Any $w \in \Zsing(T)$ is uniquely determined by the values $w(v_j)$, $\nabla w(v_j)$ and $\nabla w(\midpoint(f_j)) \cdot \nu_{f_j}$, $j = 0,1,2$, i.e., those are its degrees of freedom, see Figure~\ref{fig:zien} (right). 
  \item  
   The corresponding global space $\Zsing$ is $C^1$-conforming and is given by 
    $$ \Zsing = \lbrace w \in H^2(\Omega)\colon w|_T\in \Zsing(T)\text{ for all }T\in \tria\rbrace.$$
  \end{enumerate}
\end{lemma}

Note that $w \in \Zsing$ is continuous and its tangential derivatives are continuous on each face.  Furthermore, the normal derivatives on each face are quadratic and coincide on the end points and the mid points of each face.  Therefore, they are continuous as well.

Continuity of the normal derivatives can be achieved with a smaller local space, namely the one of the so-called reduced singular \Zienkiewicz element
\begin{equation}\label{eq:ReducedZienkiewicz}
	\Zred(T) \coloneqq \lbrace w \in \Zsing(T)\colon \nabla w|_\edge \cdot \nu_\edge \in \mathcal{P}_1(\edge)\text{ for all edges } \edge \in \edges(T)\rbrace.
\end{equation}
The functions in $\Zred(T)$ can be decomposed into quadratic functions and the functions $\lambda_j^2 \lambda_i - \lambda_i^2 \lambda_j$, $i,j \in \{0,1,2\}$ with $i \neq j$, corrected by the rational bubble function $B_{f_{k}}$, $k\in \{0,1,2\}$, such that its normal derivatives are linear on the faces.

\begin{lemma}[Reduced singular \Zienkiewicz element {\cite[Sec.~6.1.6]{Ciarlet2002}}]\label{lem:red-sing-Zien} \hfill
	\begin{enumerate}
  \item
		Any $w \in \Zred(T)$ is uniquely determined by the values $w(v_j)$ and $\nabla w(v_j)$ i.e., those are its degrees of freedom, see Figure~\ref{fig:zien} (left). 
  \item
    The corresponding global space $\Zred$ is $C^1$-conforming and is given by
		\begin{align*}
			\Zred &= \lbrace w \in H^2(\Omega)\colon w|_T\in \Zred(T)\text{ for all }T\in \tria\rbrace.
		\end{align*}
	\end{enumerate}
\end{lemma}
\subsection{\Guzman{}--Neilan element}
\label{sec:GNelement}
The \Guzman{}--Neilan element~\cite{GuzmanNeilan2014} is a mixed finite element consisting of a pressure space $\Qspace$ and a velocity space $\Vspace$. 
For those, the discretely divergence-free velocity functions are even exactly divergence-free. 
They are constructed in such a way that the divergence-free velocity functions are represented as the curl of some $C^1$-conforming \Zienkiewicz type finite element. 
With those they form an exact sequence with compatible projection operators. 

Let us summarize the construction in 2D for the lowest-order \Guzman{}--Neilan element in case the singular \Zienkiewicz element from Section~\ref{sec:Zienkiewicz} is used. 
Note that in the original contribution a slightly modified \Zienkiewicz type element is used, which we present in Remark~\ref{rmk:GN-orig} below, but this does not change any essential properties.
For functions $g\colon \RR^2 \to \RR$ we define
\begin{align}
  \label{eq:R}
  \bfcurl g
  &\coloneqq 
  \begin{pmatrix}
    \partial_{x_2} g \\ -\partial_{x_1} g
  \end{pmatrix}
  = R\, \nabla g\qquad\text{with }R \coloneqq \begin{pmatrix}
    0 & 1 \\ -1 & 0 
  \end{pmatrix}.
\end{align}
As above let~$b_T \coloneqq \lambda_0 \lambda_1 \lambda_2$ be the element bubble function, let~$b_{\edge_j} \coloneqq \lambda_{j+1} \lambda_{j+2}$ be the edge bubble functions, and let~$B_{\edge_j}$ be the rational bubble functions  as in~\eqref{eq:DefRatBubble}, for $j=0,1,2$. 
The local pressure space and velocity space are given by 
\begin{subequations}\label{eq:GN-local}
\begin{alignat}{3}
  Q(T) &\coloneqq \mathcal{P}_0(T),&\\
	V(T) &\coloneqq \mathcal{P}_1(T)^2 &&\oplus \, \linearspan\{ \bfcurl(\lambda_j^2 \lambda_{j+1} - \lambda_{j+1}^2 \lambda_j)\}_{j = 0,1,2} \\
	&& & \oplus \linearspan\{ \bfcurl B_{f_j}\}_{ j = 0,1,2 }.\notag
\end{alignat} 
\end{subequations}
Note that we have the representation (without direct sum)
\begin{align}\label{eq:VT}
	V(T) = \mathcal{P}_1(T)^2 + \bfcurl \Zsing(T). 
\end{align}
Since $V(T) \subset \mathcal{P}_2(T)^2 + \linearspan\{ \bfcurl B_{f_j}\}_{ j = 0,1,2 }$ it follows by Lemma~\ref{lem:RatBubbles}~\ref{itm:SameDerivAsBubble} that 
\begin{align}
  \label{eq:VinP2}
  v|_{\edge_j} \in \mathcal{P}_2(\edge_j)^2 \qquad \text{for any $v \in \Vspace(T)$}.
\end{align}
This ensures that using the degrees of freedom of the vector-valued second-order Lagrange finite element functions, see Figure~\ref{fig:GN}, one obtains globally continuous functions.

\begin{lemma}[\Guzman{}-Neilan element {\cite[Sec.~3.1--3.2]{GuzmanNeilan2014}}]\label{lem:GN} \leavevmode
  \begin{enumerate}
  \item 
    Any $v \in \Vspace(T)$ is uniquely determined by the values $v$ at vertices and the midpoints of the faces, see Figure~\ref{fig:GN} (left). 
    Any $q\in \Qspace(T)$ is uniquely determined by its value at the barycenter, see Figure~\ref{fig:GN} (right).
  \item  
    The global space $\Vspace$ is $C^0$-conforming and the global spaces are given by
    \begin{align*}
      \Vspace &= \lbrace v \in  H^1(\Omega)^2
      \colon v|_T\in \Vspace(T)\text{ for all }T\in \tria\rbrace,
      \\
      \Qspace &= \lbrace v \in L^2(\Omega)\colon v|_T\in \Qspace(T)\text{ for all }T\in \tria\rbrace. 
    \end{align*}
  \end{enumerate}
\end{lemma}
\begin{figure}[t]
    \begin{tikzpicture}[scale=1]
		\coordinate (z1) at (0,0);
		\coordinate (z2) at (2,0);
		\coordinate (z3) at (1,1.73);
		\fill (z1) circle (2pt);
		\fill (z2) circle (2pt);
		\fill (z3) circle (2pt);
		\fill ($(z1)!0.5!(z2)$) circle (2pt);
		\fill ($(z2)!0.5!(z3)$) circle (2pt);
		\fill ($(z3)!0.5!(z1)$) circle (2pt);
		\draw (z1) -- (z2) -- (z3) -- cycle;
		
		\coordinate (y1) at (4,0);
		\coordinate (y2) at (6,0);
		\coordinate (y3) at (5,1.73);
		\draw (y1) -- (y2) -- (y3) -- cycle;

    \fill ($(y1)!1/2!(y2)!1/3!(y3)$) circle (2pt);

  \end{tikzpicture}
  \caption{Degrees of freedom in the \Guzman{}--Neilan elements for the vector-valued velocity (left) and the pressure (right).} \label{fig:GN}
\end{figure}%

A special feature of the \Guzman{}--Neilan element is the fact that they are part of an exact discrete de~Rham complex~\cite[Eq.~4.5]{GuzmanNeilan2014}
\begin{align*}
	\mathbb{R}\; \xrightarrow{\hspace{1em}}\;	\Zsing 
	\; \xrightarrow{\bfcurl} \;\Vspace  \;\xrightarrow{\divergence }\;  \Qspace 
	\;\xrightarrow{\hspace{1em}}\; 0. 
\end{align*} 
Note that when including zero-traces the \Guzman{}--Neilan elements are inf-sup stable, as proved in~\cite{GuzmanNeilan2014} by use of a Fortin operator. 

\begin{remark}[Divergence-free functions]\label{rmk:div-free}
	The spaces of discretely divergence-free velocity functions are defined by
	\begin{align*}
		\Vspacediv& \coloneqq \Big\{v_h \in \Vspace\colon 
		\int_{\Omega} q_h \divergence v_h  \dx  
		= 0\; \text{ for all } q_h \in \Qspace\Big\}.
	\end{align*}
	Due to the representation of the velocity space in~\eqref{eq:VT} 
	we have that $\divergence \Vspace \subset \Qspace$. 
	Therefore, the discretely divergence-free velocity functions are exactly divergence-free, that is 
		$\Vspacediv \subset H^1_{\divergence}(\Omega) \coloneqq \lbrace v \in H^1(\Omega)^2 \colon \divergence v = 0\rbrace$. 
\end{remark}
\begin{remark}[Alternative space]\label{rmk:GN-orig}
	Instead of using the standard \Zienkiewicz elements (see Section~\ref{sec:Zienkiewicz}), \Guzman{} and Neilan use in~\cite{GuzmanNeilan2014} an alternative $C^1$-conforming \Zienkiewicz type element for their construction. 
	It has the same degrees of freedom, but the local velocity space reads
	\begin{equation*}
		\widetilde{\Zspace}_s(T) = \mathcal{P}_2(T) \oplus  
		\linearspan \set{ \lambda_j^2 \lambda_{j+1}}_{j = 0,1,2} \oplus 
		\linearspan \set{ B_{\edge_{j}}}_{j = 0,1,2}.
	\end{equation*}
	This means that they use the basis functions $\lambda_{j}^2\lambda_{j+1}$ instead of $\lambda_{j}^2\lambda_{j+1} -  \lambda_{j+1}^2 \lambda_{j}$, $j = 0,1,2$. 
	Then, instead of $V(T)$ as in~\eqref{eq:GN-local} they use the local spaces 
	\begin{align*}
		\widetilde{V}(T) &\coloneqq \mathcal{P}_1(T)^2 \oplus \linearspan\{ \bfcurl(\lambda^2_j \lambda_{j+1})\}_{ j = 0,1,2}	
		\oplus \linearspan\{ \bfcurl B_{f_j} \}_{ j = 0,1,2 },\\
		\widetilde{Q}(T) &\coloneqq \mathcal{P}_0(T).
	\end{align*}
	The pressure space is the same as above, but the velocity space is slightly different. 
	The corresponding global spaces form an exact sequence. 
	All the results cited in this section are proved in~\cite{GuzmanNeilan2014} for this construction, but also hold for the version of the spaces presented here. 
	This includes inf-sup stability, commuting projection properties, and the exactness of the discrete de~Rham complex. 	
\end{remark}

\subsubsection*{Reduced \Guzman{}--Neilan element}

\begin{figure}[t]
	\begin{tikzpicture}[scale=1]
		\coordinate (z1) at (0,0);
		\coordinate (z2) at (2,0);
		\coordinate (z3) at (1,1.73);
		\fill (z1) circle (2pt);
		\fill (z2) circle (2pt);
		\fill (z3) circle (2pt);
		\draw (z1) -- (z2) -- (z3) -- cycle;
		\draw (y1) -- (y2) -- (y3) -- cycle;
		\def\lengthOrthoLine{.28cm}
		\coordinate (Midpoint) at ($(z1)!0.5!(z3)$);
		\draw[thick,->] (Midpoint) -- ($(Midpoint)!+\lengthOrthoLine!90:(z3)$) 
		coordinate (OrthoEnd);
		
		\coordinate (Midpoint2) at ($(z1)!0.5!(z2)$);
		\draw[thick,->] (Midpoint2) -- ($(Midpoint2)!-\lengthOrthoLine!90:(z2)$) 
		coordinate (OrthoEnd);
		
		\coordinate (Midpoint3) at ($(z3)!0.5!(z2)$);
		\draw[thick,->] (Midpoint3) -- ($(Midpoint3)!+\lengthOrthoLine!90:(z2)$) 
		coordinate (OrthoEnd);
		
		\coordinate (y1) at (4,0);
		\coordinate (y2) at (6,0);
		\coordinate (y3) at (5,1.73);
		\draw (y1) -- (y2) -- (y3) -- cycle;
		
		\fill ($(y1)!1/2!(y2)!1/3!(y3)$) circle (2pt);
		
	\end{tikzpicture}
	\caption{Degrees of freedom in the reduced \Guzman{}--Neilan elements for the vector-valued velocity (left) and the pressure (right). Dots denote point evaluations and arrows evaluation of the normal component.}\label{fig:GNred}
\end{figure}%

\Guzman{} and Neilan have also developed a reduced element in~\cite{GuzmanNeilan2014} with fewer degrees freedom for the local velocity space, see Figure~\ref{fig:GNred}. 
The local spaces can be written with tangential unit vectors $\tau_f$ as
\begin{align*}
  \Vred(T) &\coloneqq \set{ v \in V(T)\colon v|_{\edge} \cdot\tau_{\edge} \in \mathcal{P}_1(\edge)\text{ for all  
  			$\edge \in \edges(T)$}},
  \\
  Q_r(T) &\coloneqq \mathcal{P}_0(T). 
\end{align*}
Since the tangential unit vectors can be obtained by $\tau_f = R^\top \nu_f$ with rotation matrix $R$ defined in~\eqref{eq:R}, we have for each $v\in Z(T)$ that
\begin{equation*}
\bfcurl v|_f \cdot \tau_f = R \nabla v|_f \cdot \tau_f = \nabla v|_f \cdot R^\top \tau_f = \nabla v|_f \cdot \nu_f.
\end{equation*}
Since $\nabla v|_f \cdot \nu_f \in \mathcal{P}_1(f)$ with $v\in Z(T)$ if and only if $v\in \Zsing(T)$, we obtain
\begin{align}\label{eq:Vr}
  \Vred(T)&= \mathcal{P}_1(T)^2 + \bfcurl \Zred(T).
\end{align}
 It is possible to determine an explicit basis of $\Vred(T)$, but in our implementation this is computed on the fly.
\begin{lemma}[Reduced \Guzman{}--Neilan element {\cite[Sec.~6]{GuzmanNeilan2014}}]\label{lem:GNred} \leavevmode
  \begin{enumerate}
  \item 
    Any $v \in \Vred(T)$ is uniquely determined by the values of $v$ at vertices and the values of $v \cdot \nu_f$ at the edge midpoints, see Figure~\ref{fig:GNred} (left). 
    Any $q\in \Qspace(T)$ is uniquely determined by its value at the barycenter, see Figure~\ref{fig:GNred} (right).
  \item  
    The global space $\Vred$ is $C^0$-conforming and the global spaces are given by
    \begin{align*}
      \Vred &= \lbrace v \in 
      H^1(\Omega)^2
      \colon v|_T\in \Vred(T)\text{ for all }T\in \tria\rbrace,
      \\
      \Qred &= \lbrace v \in L^2(\Omega)\colon v|_T\in \Qred(T)\text{ for all }T\in \tria\rbrace. 
    \end{align*}
  \end{enumerate}
\end{lemma}

Since for the reduced \Guzman{}--Neilan element we have $\Vred \subset \Vspace$, also the corresponding discretely divergence-free subspace of $\Vred$ is exactly divergence-free, cf.~Remark~\ref{rmk:div-free}. 

\section{Rational polynomials}
\label{sec:rat-poly}

\noindent Certain types of rational polynomials occur in the 2D singular \Zienkiewicz and in the \Guzman{}--Neilan element, as introduced in Section~\ref{sec:FEspace}. 
For the implementation of those elements no exact quadrature is available in the literature. 
Instead, \cite{Schneier2015} presents an inexact quadrature that resolves the corner singularities by use of the Duffy transformation applied to the rational function, and then uses  Gau\ss{} quadrature to approximate the integral. 
Here, in Section~\ref{subsec:IntFormula} we present exact recursive quadrature formulas for a class of rational polynomials that include the ones in Section~\ref{sec:FEspace}. 
Furthermore, we show properties of rational polynomials in Section~\ref{subsec:DifAndReg} that we exploit in the recursion.  

Throughout this section $T = [v_0,v_1,v_2]$ denotes a triangle with barycentric coordinates $\lambda = (\lambda_0,\lambda_1,\lambda_2)$ with $\lambda_0 + \lambda_1 + \lambda_2 = 1$.
For multi-indices $\alpha= (\alpha_0,\alpha_1,\alpha_2)$ and $\beta=(\beta_0,\beta_1,\beta_2)$ in $\setN^3_0$ we define the rational function
\begin{align}\label{def:rat-fct}
	\frR^\alpha_\beta \coloneqq \frR^\alpha_\beta(\lambda) \coloneqq \frac{\lambda^\alpha}{(1-\lambda)^\beta} = \frac{\prod_{j=0}^2 \lambda_j^{\alpha_j}}{\prod_{j=0}^2 (1-\lambda_j)^{\beta_j}}.
\end{align}
With this notation and  the canonical basis vectors $e_j\in \mathbb{R}^3$, the rational bubble functions $B_{f_j}$ in~\eqref{eq:DefRatBubble} read
\begin{equation*}
B_{f_{j-1}} = \frR^{(2,2,2)-e_j}_{(1,1,1)-e_j}\qquad \text{for }j=1,2,3.
\end{equation*}
\subsection{Differentiation and regularity}\label{subsec:DifAndReg}
Let $\frR^\alpha_\beta\colon T \to \mathbb{R}$ be a rational polynomial with multi-indices $\alpha,\beta \in \mathbb{N}_0^3$ and let $\hat{\frR}^\alpha_\beta \colon  \widehat{T} \to \mathbb{R}$ be the function with $\hat{\frR}^\alpha_\beta \circ \lambda = \frR_\beta^\alpha$. 
To simplify the presentation, for $j=0,1,2$ we use the notation
\begin{equation}\label{eq:DefNotationLambda}
\nabla_\lambda \frR^\alpha_\beta \coloneqq (\nabla_{\lambda} \hat{\frR}^\alpha_\beta)\circ \lambda\qquad \text{and}\qquad \partial_{\lambda_j} \frR^\alpha_\beta  \coloneqq (\partial_{\lambda_j} \hat{\frR}^\alpha_\beta) \circ \lambda.
\end{equation}
The chain rule relates the gradients $\nabla_x \frR^\alpha_\beta \colon T\to \setR^2$ and $\nabla_\lambda \hat{\frR}^\alpha_\beta \colon \widehat{T} \to \mathbb{R}^3$ via 
\begin{equation}\label{eq:trafogradx}
\nabla_x \frR^\alpha_\beta
= 	\nabla_x (\hat{\frR}^\alpha_\beta \circ \lambda) 
= \nabla_{x} \lambda \,(\nabla_{\lambda} \hat{\frR}^\alpha_\beta)\circ \lambda = \nabla_{x} \lambda \,\nabla_{\lambda} \frR^\alpha_\beta.
\end{equation} 
This extends  to vector-valued functions in the obvious manner.
Furthermore, the Hessian matrix of $\frR^\alpha_\beta  $ reads 
\begin{equation}\label{eq:chainrule-2}
	\nabla_x^2 \frR^\alpha_\beta  
	= \nabla_x \lambda\, ((\nabla_{\lambda}^2 \hat{\frR}^\alpha_\beta)\circ \lambda )  (\nabla_x \lambda )^\top 
	= \nabla_x \lambda\, (\nabla_{\lambda}^2 \frR^\alpha_\beta)\,  (\nabla_x \lambda )^\top.
\end{equation}
These formulae allow us to differentiate rational polynomials, once we know $\nabla_x \lambda$ and $\partial_{\lambda_j} \frR^\alpha_\beta$. 
The computation of $\nabla_x \lambda$ is discussed in~\eqref{eq:nablaLam} below, and differentiation with respect to the barycentric coordinates is subject of the following lemma. 
\begin{lemma}[Multiplication and differentiation]
  \label{lem:Rmult-diff}
  For~$\alpha,\beta,\sigma,\tau \in \setN_0^3$  we have
  \begin{align}
    \frR^\alpha_\beta \cdot
     \frR^\sigma_\tau &= \frR^{\alpha+\sigma}_{\beta+\tau}, \label{eq:MultiplyRat}
    \\
    \partial_{\lambda_j} \frR^\alpha_\beta
    &= \alpha_j
    \frR^{\alpha-e_j}_\beta - \beta_j \frR^\alpha_{\beta+e_j} \quad \text{ for } j \in \{0,1,2\}. \label{eq:diffRat}
  \end{align}
\end{lemma}
\begin{proof}
  The first identity follows directly by the definition of  $\frR^\alpha_\beta$ in~\eqref{def:rat-fct}. 
  The second identity is a consequence of the quotient rule.
\end{proof}

\begin{lemma}[Regularity]\label{lem:regularity}
	For $m \in \setN_0$ and $\alpha,\beta \in \setN^3$ we have $\frR^\alpha_\beta \in W^{m,p}(T)$ for $p \in [1,\infty)$ if and only if 
	$	\abs{\alpha} - \norm{\alpha + \beta}_{\infty}> m - {2}/{p}$, and 
	for $p  = \infty$ if and only if 
	$	\abs{\alpha} - \norm{\alpha + \beta}_{\infty}\geq m$.
	In particular, one has 
	\begin{enumerate}[label = (\roman*)]
	\item \label{itm:Wmpsmall}  $\frR^\alpha_\beta \in W^{m,p}(T)$ for $p \in [1,2)$ if 
    and only if
    $m-1 \leq \abs{\alpha} - \norm{\alpha+\beta}_\infty $,
  \item \label{itm:Wmplarge}  $\frR^\alpha_\beta \in W^{m,p}(T)$ for $p\in [2,\infty]$ if and only if $ m \leq \abs{\alpha} - \norm{\alpha+\beta}_\infty$.
	\end{enumerate}
\end{lemma}
\begin{proof}
	Let us start proving the case for $m = 0$. 
	
	Let $p\in [1,\infty)$. 
	Note that the singularities occur only in vertices and hence it suffices to verify integrability in the neighborhoods of vertices. 
	By symmetry it suffices to consider the vertex $v_0$ which is without loss of generality at the origin $v_0 = 0$. 
	Since the terms $\lambda_0$, $(1-\lambda_1)$, and $(1 - \lambda_2)$ are equivalent to a constant in a neighborhood of the origin we obtain for small values $\varepsilon >0$
	\begin{align*}
    \begin{aligned}
    		\int_{\{x \in T\colon \abs{x} \leq \epsilon\}} \left(\frac{\lambda^\alpha}{ (1-\lambda)^\beta}\right)^p \dx& \eqsim \int_{\{x \in T\colon \abs{x} \leq \epsilon\}} \left(\frac{\lambda_1^{\alpha_1}\lambda_2^{\alpha_2}}{ (1-\lambda_0)^{\beta_0}}\right)^p \dx \\
		&\eqsim \int_0^{\epsilon} r^{p(\alpha_1+\alpha_2-\beta_0)} r \dr.
    \end{aligned}
	\end{align*}
	The integral is finite if and only if $p(\alpha_1+\alpha_2 - \beta_0) + 1 > -1$. 
	This condition is equivalent to $\alpha_1+\alpha_2 - \beta_0  > -2/p$ or in other words $\abs{\alpha}-\alpha_0-\beta_0 > -2/p$. 
	By symmetry in fact we require $\abs{\alpha}-\norm{\alpha + \beta}_{\infty} > -2/p$, which proves the claim.
	Similarly, for $p = \infty$, in the neighborhood of $v_0 = 0$, we have that 
	\begin{align*}
		\frac{\lambda^\alpha}{ (1-\lambda)^\beta} \eqsim \frac{\lambda_1^{\alpha_1}\lambda_2^{\alpha_2}}{ (1-\lambda_0)^{\beta_0}} \lesssim r^{\alpha_1+\alpha_2-\beta_0}
	\end{align*}
	is bounded, if $\alpha_1+\alpha_2-\beta_0 \geq 0$.   This is equivalent to $\abs{\alpha} - \alpha_0 - \beta_0 \geq 0$, which proves the claim for $p = \infty$. 
	This estimate is sharp if we consider a path from the barycenter to the vertex~$v_0=0$.
	
	We proceed by induction over~$m$. For $j = 0,1,2$ the identity~\eqref{eq:diffRat} states
	\begin{align}
    \label{eq:regularity-aux1}
		\partial_{\lambda_j} \frR^\alpha_\beta = \alpha_j\frR^{\alpha-e_j}_\beta -\beta_j\frR^{\alpha}_{\beta+e_j}.
	\end{align}
  	The terms vanish for~$\alpha_j=0$ or $\beta_j=0$, respectively. 
  	Otherwise, for the corresponding indices we obtain
	\begin{align*}
    \abs{\alpha-e_j} -\norm{\alpha-e_j + \beta}_\infty &\geq
    \abs{\alpha-e_j} -\norm{\alpha+\beta}_\infty =
    \abs{\alpha} -\norm{\alpha + \beta}_\infty-1,
    \\
    \abs{\alpha} -\norm{\alpha + \beta+e_j}_\infty &\geq
    \abs{\alpha} -\norm{\alpha+\beta}_\infty -1.
  \end{align*}
  In both cases 
  $p \in [1,\infty)$ and $p = \infty$
   we have $\frR^{\alpha-e_j}_\beta,\frR^{\alpha}_{\beta+e_j} \in W^{m-1,p}(T)$ for $p\in[1,\infty]$ by induction hypothesis and the formula in~\eqref{eq:trafogradx}.
  Hence, $\frR^\alpha_\beta \in W^{m,p}(T)$ follows by~\eqref{eq:trafogradx} and~\eqref{eq:regularity-aux1}. The sharpness follows from the fact that the singularities of~$\frR^{\alpha-e_j}_\beta$ and $\frR^{\alpha}_{\beta+e_j}$ are linearly independent and hence cannot cancel.  
  
  The statements in \ref{itm:Wmpsmall} and \ref{itm:Wmplarge} are an immediate consequence.
\end{proof}

\subsection{Integration formula}\label{subsec:IntFormula}
We are now in the position to derive recursive formulas for the (mean) integral of the rationals functions~$\frR^\alpha_\beta$. Recall, that~$T \subset \RR^2$ is a triangle. We define for $\alpha,\beta \in \setN^3_0$
\begin{align}\label{eq:Integral}
  \mathcal{I}(\alpha,\beta)
  &\coloneqq \dashint_T \frR^\alpha_\beta(\lambda(x))\dx = \dashint_T \frac{\lambda^\alpha}{(1-\lambda)^\beta}\dx. 
\end{align}
Note that  by affine transformation this definition is independent of~$T$, so without loss of generality we may use
\begin{align}
  \label{eq:defT}
  T = [ (0,0)^\top,(1,0)^\top,(0,1)^\top].
\end{align}
By Lemma~\ref{lem:regularity} we have
\begin{align}
  \label{eq:I-finite}
  \mathcal{I}(\alpha,\beta) < \infty\qquad\text{if and only if}\qquad	\lVert \alpha +\beta \rVert_\infty \leq \abs{\alpha} +1.
\end{align}
\subsubsection{\texorpdfstring{Special case $\beta_0=\beta_1=0$}{Special case beta0=beta1=0}}\label{subsubsec:SpecialCase}
The following lemma determines the value $\mathcal{I}(\alpha,0)$, i.e., the special case of polynomial integrand, as 
  \begin{align}\label{eq:I-poly}
	\mathcal{I}(\alpha,0) &=2 \frac{\alpha!}{(2+\abs{\alpha})!} = 2 \frac{\alpha_0!\alpha_1!\alpha_2!}{(2 + \alpha_0 + \alpha_1+\alpha_2)!}.
\end{align}

\begin{lemma}[Polynomials]\label{lem:QuadPolynomials}
  For any $d$-simplex $T$ with $d \in \mathbb{N}$ and multi-index $\alpha = (\alpha_0, \ldots, \alpha_d)\in \setN^{d+1}_0$, one has the identity 
  \begin{align*}
    \dashint_{T} \lambda^{\alpha} \dx  &=
    \frac{ d! \alpha!}{(d + \abs{\alpha})!}
    =  \frac{d! \alpha_0!\alpha_1!\alpha_2!\dots \alpha_d!}{(d + \alpha_0 + \alpha_1+\alpha_2+ \dots + \alpha_d)!}.
  \end{align*}
\end{lemma}
\begin{proof}
  This known identity is for example presented in~\cite[p.~222]{Stroud1971} for~$\alpha_0=0$. The general case with proof can be found in~\cite[Eq.~2.3]{GrundmannMoeller78}. 
  For the convenience of the reader we include a short alternative proof for polynomials on a $d$-simplex, for dimension $d\in \mathbb{N}$, and $\alpha \in \mathbb{N}^{d+1}_0$. 
  We also apply the same technique later in Proposition~\ref{pro:intrational-new}. 
  Let~$T$ be as in~\eqref{eq:defT}. Using the Gamma function $\Gamma(s) = \int_0^\infty t^{s-1} \exp(-t)\dt$, which satisfies $\Gamma(s) = (s-1)!$ for $s \in \mathbb{N}$, we obtain
  \begin{align*}
  	\lefteqn{\Gamma(d+\abs{\alpha}+1) 
  		\dashint_T \lambda(x)^\alpha\dx } \qquad \\
  	&= d!\int_0^\infty \int_T (1-x_1-x_2 - \dots - x_d)^{\alpha_0} x_1^{\alpha_1} x_2^{\alpha_2} \cdots x_d^{\alpha_d} \, t^{d+\abs{\alpha}} \exp(-t) \dx \dt.
  \end{align*}
  The substitution $(y_0,y_1,\ldots, y_d ) 
  = t(1-x_1-\dots - x_d, x_1,\dots,x_d)$ leads to
  \begin{align*}
  	\lefteqn{\Gamma(d+\abs{\alpha}+1) \dashint_T \lambda(x)^\alpha\dx } \qquad &
  	\\
  	&=d!\,  \int_0^\infty  \dots \int_0^\infty y_0^{\alpha_0} \dotsm y_d^{\alpha_d} \exp(-y_0-\dots - y_d)\dy_0 \dy_1 \dots \dy_d \\
  	& =d!\, \prod_{j = 0}^{d} \int_0^\infty z^{\alpha_j} \exp(-z) \dz	
  	 =d!\, \prod_{j = 0}^{d} \Gamma(\alpha_j + 1) = d!\alpha!. 
  \end{align*}
  This proves the claim. 
\end{proof}

The following lemma reveals another special case with explicit value of $\mathcal{I}(\alpha,\beta)$ generalizing the previous lemma with $\beta = (0, 0, \beta_2)$ for some $\beta_2 \in \mathbb{N}_0$ and $\alpha \in \mathbb{N}_0^3$. 

\begin{lemma}\label{lem:QuadRat-1}
For any triangle $T$, 
 for multi-indices $\alpha = (\alpha_0, \alpha_1, \alpha_2)\in \setN^{3}_0$, and $\beta = (0,0,\beta_2)$ with $\beta_2 \in \mathbb{N}_0$ such that $\norm{\alpha + \beta}_{\infty} \leq \abs{\alpha} + 1$, one has the identity  
\begin{equation}\label{eq:I-rat-special}
	\dashint_{T} \frac{\lambda^{\alpha}}{(1 - \lambda)^\beta} \dx    
	=  2\frac{\alpha_0! \alpha_1!\alpha_2!}{(\abs{\alpha}-\beta_2 +2)!} \frac{(\alpha_0+\alpha_1+1-\beta_2)!}{(\alpha_0+\alpha_1+1)!}.
\end{equation}
\end{lemma}
\begin{proof}
We proceed similarly as in Lemma~\ref{lem:QuadPolynomials} using the Gamma function~$\Gamma(s) = \int_0^\infty t^{s-1} \exp(-t)\,dt$ with $\Gamma(s+1)=s!$ for $s \in \setN_0$. Let~$T$ be as in~\eqref{eq:defT}.  
By the assumption $\norm{\alpha + \beta}_{\infty} \leq \abs{\alpha} + 1$ the integral is finite, see Lemma~\ref{lem:regularity}. 
With this and the fact that $\beta_1 = 0$ it follows that 
\begin{align*}
2+\abs{\alpha}-\beta_2 +1 \geq 1 + \norm{\alpha+\beta}_\infty - \beta_2+1 \geq 2.
\end{align*}
For this reason $\Gamma(2 + |\alpha| - \beta_2  +1)$ is well-defined. 
Employing the substitution $(y_0,y_1,y_2) = t(1-x_1- x_2, x_1,x_2)$ we obtain
\begin{align*}
	\mathrm{I} &\coloneqq \Gamma(2+\abs{\alpha}-\beta_2+1) \dashint_T \frac{\lambda(x)^\alpha}{(1-\lambda_2(x))^{\beta_2}}\dx
	\\
	&= 2 \int_0^\infty \int_T \frac{(1-x_1-x_2)^{\alpha_0} x_1^{\alpha_1} x_2^{\alpha_2}}{(1-x_2)^{\beta_2}} t^{2+\abs{\alpha}-\beta_2} \exp(-t) \dx \dt
	\\
	&= 2\int_0^\infty \int_0^\infty \int_0^\infty \frac{y_0^{\alpha_0} y_1^{\alpha_1} y_2^{\alpha_2}}{(y_0+y_1)^{\beta_2}} \exp(-y_0-y_1-y_2) \dy_0 \dy_1 \dy_2
	\\
	&= 2\alpha_2! \int_0^\infty \int_0^\infty \frac{y_0^{\alpha_0} y_1^{\alpha_1}}{ (y_0+y_1)^{\beta_2}} \exp(-y_0-y_1)\dy_0\dy_1.
\end{align*}
We substitute $s = y_0+y_1$, and then $z = y_0/s$ and use Lemma~\ref{lem:QuadPolynomials} for~$n=1$ to find
\begin{align*}
	\mathrm{I}
	&= 2\alpha_2! \int_0^\infty \int_0^s \frac{y_0^{\alpha_0} (s-y_0)^{\alpha_1}}{ s^{\beta_2}} \exp(-s)\dy_0\ds
	\\
	&= 2\alpha_2! \int_0^\infty \int_0^1 z^{\alpha_0} (1-z)^{\alpha_1} \, \dz\, s^{\alpha_0+\alpha_1-\beta_2+1}  \,\exp(-s) \ds
	\\
	&= 2\alpha_2! \frac{\alpha_0! \alpha_1!}{(\alpha_0+\alpha_1+1)!} (\alpha_0+\alpha_1+1-\beta_2)!.
\end{align*}
With $\Gamma(s) = (s-1)!$ for $s \in \mathbb{N}$ this proves the claim.
\end{proof}

\subsubsection{\texorpdfstring{Special case $\alpha_0=\beta_0=0$}{Special case alpha0=beta0=0}}

	In case that~$\alpha_0=\beta_0=0$ the integral $\mathcal{I}(\alpha,\beta)$ can be split by Fubini's theorem, leading to
	\begin{equation}\label{def:J}
		\begin{aligned}
			\mathcal{J}(\alpha_1,\alpha_2,\beta_1,\beta_2)& \coloneqq   \mathcal{I}((0,\alpha_1,\alpha_2),(0,\beta_1,\beta_2)) 
			\\
			&\hphantom{:}=
			2	\int_0^1 \frac{y^{\alpha_2}}{(1-y)^{\beta_2}} \int_0^{1-y}  \frac{x^{\alpha_1}}{(1-x)^{\beta_1}} \dx \dy.
		\end{aligned}
	\end{equation}
	To evaluate~$\mathcal{J}(\alpha_1,\alpha_2,\beta_1,\beta_2)$ we denote the inner integral by
	\begin{align}\label{def:Jx}
		\mathcal{J}^x(\alpha_1,\beta_1,y)& \coloneqq \int_0^{1-y} \frac{x^{\alpha_1}}{(1-x)^{\beta_1}} \dx\qquad \text{for } y\in (0,1).
	\end{align}
Hence, we have the identity 
	\begin{align}
		\label{eq:JvsJy}
		\mathcal{J}(\alpha_1,\alpha_2,\beta_1,\beta_2) = 
		2
		\int_0^1 \frac{y^{\alpha_2}}{(1-y)^{\beta_2}} \mathcal{J}^x(\alpha_1,\beta_1,y)\dy.
	\end{align}

	\begin{lemma}
		\label{lem:integralI}
		We have for all $\alpha_1,\beta_1 \in \setN_0$ and $y\in (0,1)$
		\begin{align*}
			\mathcal{J}^x(\alpha_1,\beta_1,y)=
			\begin{cases}
				{\displaystyle \frac{(1-y)^{\alpha_1 + 1}}{\alpha_1 + 1}}
				&\text{if }\beta_1 = 0,\\[1mm]
				{\displaystyle - \sum_{j=1}^{\alpha_1} \tfrac{1}{j}(1-y)^j-\log(y)} 
				&\text{if }\beta_1 = 1,\\[1mm]
				{\displaystyle
					\frac{\beta_1 - \alpha_1 - 2}{\beta_1 - 1}\mathcal{J}^x(\alpha_1,\beta_1\!-\!1,y) + 
					\frac{1}{\beta_1-1} \frac{(1\!-\!y)^{\alpha_1 + 1}}{y^{\beta_1-1}} } &\text{if }\beta_1 > 1.
			\end{cases}
		\end{align*}
	\end{lemma}
	\begin{proof}
		The first identity follows from the integration of polynomials in 1D.
		We continue with~$\beta_1=1$. 
		In case $\alpha_1 = 0$ we have
		\begin{align*}
			\mathcal{J}^x(0,1,y)= - \int_0^{1-y} \frac{1}{x-1} \dx = - \log(y).
		\end{align*}
		For $\alpha_1 > 0$ the formula for geometric sums leads to
		\begin{align*}
			\mathcal{J}^x(0,1,y)&=\int_0^{1-y} \frac{x^{\alpha_1}}{1-x} \dx = - \int_0^{1-y} \frac{x^{\alpha_1} - 1}{x-1} \dx - \int_0^{1-y} \frac{1}{x-1} \dx \\
			&= - \sum_{j=0}^{\alpha_1-1}\int_0^{1-y}  x^j \dx - \int_0^{1-y} \frac{1}{x-1} \dx 
			 = -\sum_{j=1}^{\alpha_1} \frac{1}{j}(1-y)^j - \log(y).
		\end{align*}
		Finally, let $\beta_1 > 1$. 
		We use the identity  
		\begin{align*}
			\mathcal{J}^x(\alpha_1,\beta_1,y) 
			&=   \int_0^{1-y}\! \frac{x^{\alpha_1}}{(1-x)^{\beta_1}}\dx
			 = \int_0^{1-y}\!\! \frac{ x^{\alpha_1}}{(1-x)^{\beta_1-1}}\dx + \int_0^{1-y}\! \frac{x^{\alpha_1+1}}{(1-x)^{\beta_1}}\dx
			\\
			&=     \mathcal{J}^x(\alpha_1,\beta_1-1,y) + \mathcal{J}^x(\alpha_1+1,\beta_1,y).
		\end{align*}
		Integration by parts yields for the second term
		\begin{align*}
			\mathcal{J}^x(\alpha_1+1,\beta_1,y) &=
			\int_0^{1-y} \frac{x^{\alpha_1+1}}{(1-x)^{\beta_1}}\dx
			\\
			&= \bigg[ \frac{x^{\alpha_1+1}(1-x)^{1-\beta_1}}{\beta_1-1}\bigg]^{1-y}_0 - \int_0^{1-y} \frac{\alpha_1+1}{\beta_1-1} \frac{x^{\alpha_1}}{(1-x)^{\beta_1-1}}\dx
			\\
			&= \frac{(1-y)^{\alpha_1+1} y^{1-\beta_1}}{\beta_1-1} - \frac{\alpha_1+1}{\beta_1-1}
			\mathcal{J}^x(\alpha_1,\beta_1-1,y).
		\end{align*}
		Combining theses identities concludes the proof.
	\end{proof}
	
		Since expressions for $\mathcal{J}^x$ are available by Lemma~\ref{lem:integralI}, certain integrals in $y$ still have to be computed in order to evaluate the integral $\mathcal{J}$ in~\eqref{eq:JvsJy}. 
	
	\begin{lemma}\label{lem:IntegralsII}
		Let $\alpha_2,\gamma\in \setN_0$, then we have with polygamma function $\psi^{(1)}$ that
		\begin{subequations}
			\begin{align}
				\int_0^1  y^{\alpha_2} (1-y)^\gamma \dy & = \frac{\alpha_2! \gamma!}{(\alpha_2 + \gamma + 1)!},\label{eq:FubiniIntegrals2}\\
				\int_0^1 y^{\alpha_2} \log(y) \dy & = -\frac{1}{(\alpha_2+1)^2},\label{eq:FubiniIntegrals3}\\
				\int_0^1 \frac{y^{\alpha_2}}{1-y} \log(y) \dy & = -\psi^{(1)}(\alpha_2+1) 
				= - \frac{\pi^2}{6} + \sum_{j=1}^{\alpha_2} \frac{1}{j^2}.\label{eq:FubiniIntegrals4}
			\end{align}
		\end{subequations}
	\end{lemma}
	\begin{proof}
		The formula in~\eqref{eq:FubiniIntegrals2} is the 1D version of Lemma~\ref{lem:QuadPolynomials}.
		Integration by parts leads to~\eqref{eq:FubiniIntegrals3}, that is,
		\begin{align*}
			\int_0^1 y^{\alpha_2} \log y\dy
			&= \bigg[\frac{y^{\alpha_2+1}}{\alpha_2+1} \log y\bigg]^1_0 
			- \int_0^1 \frac{y^{\alpha_2}}{\alpha_2+1}\dy = -\frac{1}{(\alpha_2+1)^2}.
		\end{align*}
		A substitution with $y = \exp(-t)$ and the properties of the polygamma function~\cite[eq.~6.4.1 and 6.4.10]{AbramowitzStegun1964} yield
		\begin{align*}
			\int_0^1 \frac{y^{\alpha_2}\log(y)}{y-1}\dy
			&= - \int_0^\infty \frac{t \exp(-\alpha_2t)}{1-e^{-t}}\dt =  - \psi^{(1)}(\alpha_2+1)
			=\frac{\pi^2}{6}- \sum_{j=1}^{\alpha_2} \frac{1}{j^2}.\qedhere
		\end{align*}
	\end{proof}	

	Combining Lemmas~\ref{lem:integralI} and~\ref{lem:IntegralsII} allows us to evaluate $\mathcal{J}$. 
	This is realized in Algorithm~\ref{algo:ComputeJ}, which will be verified in Theorem~\ref{thm:Algo-J} below.

	\begin{figure}[ht] 
		\begin{algorithm}[H] 
			\caption{Computation of $\mathcal{J}(\alpha_1,\alpha_2,\beta_1,\beta_2)$}
			\label{algo:ComputeJ}
			\SetAlgoLined%
			\SetKwFunction{KwComputeJ}{ComputeJ}
			\KwComputeJ$(\alpha_1,\alpha_2,\beta_1,\beta_2)$
			
			\KwIn{Indices $\alpha_1,\alpha_2,\beta_1,\beta_2 \in \setN_0$}
			\KwOut{Integral mean $\mathcal{J}(\alpha_1,\alpha_2,\beta_1,\beta_2)$} \If{$\max \set{\alpha_1+\beta_1,\alpha_2+\beta_2} > \alpha_1+\alpha_2+1$}{
				\Return{$\infty$}} 
			
			Sort $((\alpha_j,\beta_j))_{j=1,2}$ such that $\beta_2 \geq \beta_1$
			
			\If{$\beta_1 = 0$} { 
				\Return{$ \tfrac{2}{\alpha_1 + 1} \tfrac{\alpha_2! (\alpha_1 - \beta_2 + 1)! }{(\alpha_1 +\alpha_2 - \beta_2 + 2)!}$} } 
			\If{$\beta_1 = 1$} { 
				\If{$\beta_2 = 1$} { 
					\Return{$- 2 \sum_{i=1}^{\alpha_2} \tfrac{1}{i^2} + \tfrac{\pi^2}{3} -2 \sum_{j=1}^{\alpha_1} \tfrac 1j \tfrac{\alpha_2! (j-1)!}{(\alpha_2 + j)!}$} } 
				\Else { 
					\Return{$\tfrac{\beta_2 - \alpha_2 - 2}{\beta_2 - 1} \KwComputeJ(\alpha_1,\alpha_2,1,\beta_2-1) + \tfrac{2}{\beta_2-1} \tfrac{(\alpha_1 - \beta_2 +1)!\alpha_2!}{(\alpha_1 - \beta_2 + \alpha_2 +2)!}$} } } 
			\Else { 
				\Return{$\tfrac{\beta_1 - \alpha_1 - 2}{\beta_1 - 1} \KwComputeJ(\alpha_1,\alpha_2,\beta_1 - 1,\beta_2) + \tfrac{2}{\beta_1-1} \tfrac{(\alpha_2 - \beta_1 +1)!(\alpha_1 - \beta_2 + 1)!}{(\alpha_2 - \beta_1 + \alpha_1 - \beta_2 + 3)!}$} } 
		\end{algorithm}
	\end{figure}
	
	\begin{theorem}[Computation of $\mathcal{J}$ in Algorithm~\ref{algo:ComputeJ}]\label{thm:Algo-J}
		Let $\alpha_1,\alpha_2,\beta_1,\beta_2 \in \setN_0$. 
		Then Algorithm~\ref{algo:ComputeJ} terminates and returns $ \mathcal{J}(\alpha_1,\alpha_2,\beta_1,\beta_2)$.
	\end{theorem}
	\begin{proof}
		Due to Lemma~\ref{lem:regularity} we have~$\mathcal{J}(\alpha_1,\alpha_2,\beta_1,\beta_2)=\infty$ if and only if we have $\max \set{\alpha_1+\beta_1,\alpha_2+\beta_2} > \alpha_1+\alpha_2+1$.  
		Let us consider the remaining cases.
		Since~$\mathcal{J}(\alpha_1,\alpha_2,\beta_1,\beta_2) = \mathcal{J}(\alpha_2,\alpha_1,\beta_2,\beta_1)$ we can assume that~$\beta_2 \geq \beta_1$.
		Recall
		\begin{align*}
			\mathcal{J}(\alpha_1,\alpha_2,\beta_1,\beta_2) 
			= 2
			\int_0^1 \frac{y^{\alpha_2}}{(1-y)^{\beta_2}} \mathcal{J}^x(\alpha_1,\beta_1,y)\dy.
		\end{align*}
		The values for~$\mathcal{J}^x$ are available by Lemma~\ref{lem:integralI} and the integrals in~$y$ are computed in Lemma~\ref{lem:IntegralsII}. 
		Hence, one has
		\begin{align*}
			\mathcal{J}(\alpha_1,\alpha_2,0,\beta_2)
			&= \frac{2}{\alpha_1 + 1} \int_0^1 \frac{y^{\alpha_2}}{(1-y)^{\beta_2}} (1-y)^{\alpha_1+1}\dy \\
			& = \frac{2}{\alpha_1 + 1} \frac{\alpha_2! (\alpha_1+1 - \beta_2)! }{(\alpha_2 +\alpha_1 - \beta_2 + 2)!}.
		\end{align*}
		This verifies Algorithm~\ref{algo:ComputeJ} for $\beta_2\geq \beta_1 = 0$. 
		We continue with $\beta_2 \geq \beta_1 \geq 1$. 
		We calculate~\eqref{eq:FubiniIntegrals4} which yields
		\begin{align*}
			\mathcal{J}(\alpha_1,\alpha_2,1,1)& = - 2 \int_0^1  \frac{y^{\alpha_2}}{1-y} \log(y) \dy
			-2 \sum_{j=1}^{\alpha_1} \frac{1}{j} \int_0^1 \frac{y^{\alpha_2}}{1-y} (1-y)^j \dy\\
			&= - 2\sum_{i=1}^{\alpha_2} \frac{1}{i^2} + \frac{\pi^2}{3} -2\sum_{j=1}^{\alpha_1} \frac{1}{j}\frac{\alpha_2! (j-1)!}{(\alpha_2 + j)!}.
		\end{align*}
		This verifies  Algorithm~\ref{algo:ComputeJ} for $\beta_1=\beta_2=1$.
		
To cover the remaining cases, we exploit 	Lemma~\ref{lem:integralI} to conclude for $\beta_1 >1$ that
		\begin{align*}
		\mathcal{J}(\alpha_1,\alpha_2,\beta_1,\beta_2) 
			&= 2 \frac{\beta_1 - \alpha_1 - 2}{\beta_1 - 1} \int_0^1 \frac{y^{\alpha_2}}{(1-y)^{\beta_2}} \mathcal{J}^x(\alpha_1,\beta_1-1,y) \dy\\
			&\qquad 
			+ \frac{2}{\beta_2-1} \int_0^1 \frac{y^{\alpha_2}}{(1-y)^{\beta_2}}  \frac{(1-y)^{\alpha_1+1}}{y^{\beta_1-1}} \dy \\
			&=\frac{\beta_1 - \alpha_1 - 2}{\beta_1 - 1}  \mathcal{J}(\alpha_1,\alpha_2,\beta_1-1,\beta_2) \\
			&\qquad
			+ \frac{2}{\beta_1-1} \frac{(\alpha_2 - \beta_1 +1)!(\alpha_1 - \beta_2 + 1)!}{(\alpha_2 - \beta_1 + \alpha_1 - \beta_2 + 3)!}.
		\end{align*}
		By symmetry, we obtain for $\beta_2 > 1$ the identity
	\begin{align*}
	\mathcal{J}(\alpha_1,\alpha_2,\beta_1,\beta_2) 
	&
	=\frac{\beta_2 - \alpha_2 - 2}{\beta_2 - 1}  \mathcal{J}(\alpha_1,\alpha_2,\beta_1,\beta_2-1) \\
		&\qquad	+ \frac{2}{\beta_2-1} \frac{(\alpha_1 - \beta_2 +1)!(\alpha_2 - \beta_1 + 1)!}{(\alpha_1 - \beta_2 + \alpha_2 - \beta_1 + 3)!}.
	\end{align*}
	This yields in particular the formula for $\beta_1 = 1 < \beta_2$ and $1 < \beta_1 \leq \beta_2$ in the algorithm. 
	Since in both cases the sum  $\beta_1 + \beta_2$ is reduced by one, the recursive algorithm reaches after a finite number of steps the case $\beta_1 = 1 = \beta_2$ and terminates.
	\end{proof}

\subsubsection{General case}

In this subsection we introduce a recursive formula that successively reduces the value of~$\abs{\beta}$ or~$\alpha_0$ until the integrand~$\mathcal{I}(\alpha,\beta)$ defined in~\eqref{eq:Integral} matches one of the special cases outlined in~\eqref{eq:I-poly}, in~\eqref{eq:I-rat-special} and in Section~\ref{subsubsec:SpecialCase}. 
The reduction relies on the following recursive relations, which use the basis vectors $\hat{v}_0 = (1,0,0)$, $\hat{v}_1 = (0,1,0)$, and $\hat{v}_2 = (0,0,1)$. 

\begin{lemma}[Recursion formulae]\label{lem:rat-rec}
For any $\alpha,\beta \in \mathbb{N}^3_0$ we have the following identities for rational polynomials as defined in~\eqref{def:rat-fct}: 
\begin{enumerate}
\item \label{itm:rec-4term-a}
$\frR^{\alpha}_{\beta} = \frac12 (
\frR^{\alpha}_{\beta- \hat v_0} +
\frR^{\alpha}_{\beta-\hat v_1} +
\frR^{\alpha}_{\beta- \hat v_2})$,
\item \label{itm:rec-3term}
$\frR^{\alpha}_{\beta} =
\frR^{\alpha-\hat v_0}_{\beta-\hat v_2} -
\frR^{\alpha-\hat v_0+\hat v_1}_{\beta}$,
\item \label{itm:rec-6term}
 $ \frR^\alpha_\beta = \tfrac 12 \big( \frR^{\alpha}_{\beta-\hat v_1} + \frR^{\alpha}_{\beta- \hat v_2} + \frR^{\alpha-\hat v_0+\hat v_1}_{\beta-\hat v_1} + \frR^{\alpha-\hat v_0+\hat v_2}_{\beta-\hat v_2}\big) - \frR^{\alpha-\hat v_0+\hat v_1+\hat v_2}_{\beta}.$
\end{enumerate}
\end{lemma}
\begin{proof}
The proof of~\ref{itm:rec-4term-a} follows from 
\begin{align*}
	(1-\lambda)^{\hat{v}_0}  + (1-\lambda)^{\hat{v}_1} + (1-\lambda)^{\hat{v}_2} = (1-\lambda_0) + (1-\lambda_1) + (1-\lambda_2) = 3 - 1 = 2. 
\end{align*}
The claim in~\ref{itm:rec-3term} follows from the identity
\begin{align*}
	\frac{\lambda^\alpha}{(1-\lambda)^\beta}
	&=
	\frac{\lambda^{\alpha-\hat{v}_0}(1-\lambda_1-\lambda_2)}{(1-\lambda)^\beta}
	=\frac{\lambda^{\alpha-\hat{v}_0}}{(1-\lambda)^{\beta-\hat{v}_1}}
	-\frac{\lambda^{\alpha-\hat{v}_0+\hat{v}_2}}{(1-\lambda)^\beta}.
\end{align*}
Finally, the proof of~\ref{itm:rec-6term} is a consequence of
\begin{align*}
	\lefteqn{\frac 12 \Big(
		\frac{\lambda^\alpha}{(1-\lambda)^{\beta-\hat{v}_1}} +
		\frac{\lambda^\alpha}{(1-\lambda)^{\beta-\hat{v}_2}}+ 
		\frac{\lambda^{\alpha-\hat{v}_0+\hat{v}_1}}{(1-\lambda)^{\beta-\hat{v}_1}}+
		\frac{\lambda^{\alpha-\hat{v}_0+\hat{v}_2}}{(1-\lambda)^{\beta-\hat{v}_2}} \Big)
		- \frac{\lambda^{\alpha-\hat{v}_0+\hat{v}_1+\hat{v}_2}}{(1-\lambda)^{\beta}} }\quad
	&
	\\
	&= \frac{\lambda^{\alpha-\hat{v}_0}}{2(1-\lambda)^\beta} \big(\lambda_0 (1-\lambda_1) + \lambda_0 (1-\lambda_2) + \lambda_1 (1-\lambda_1) + \lambda_2(1-\lambda_2) - 2\lambda_1\lambda_2\big)
	\\
	&= \frac{\lambda^{\alpha-\hat{v}_0}}{2(1-\lambda)^\beta} 2\lambda_0=
	\frac{\lambda^{\alpha}}{(1-\lambda)^{\beta}}. \qedhere
\end{align*}
\end{proof}

Those recursion relations allow us to derive recursion relations for $\mathcal{I}(\alpha,\beta)$.
\begin{proposition}[Reduction]
	\label{pro:intrational-new}
	Let $T\subset \RR^2$ be a triangle and let $\alpha,\beta \in \setN_0^3$ be multi-indices such that $\mathcal{I}(\alpha,\beta)<\infty$.
	\begin{enumerate}
		\item \label{itm:intrational-symmetry}
		For any permutation $\sigma$ of $\set{0,1,2}$ we have
		\begin{align*}
			\mathcal{I}(\alpha,\beta) &= \mathcal{I}(\sigma(\alpha), \sigma(\beta)).
		\end{align*}
		\item \label{itm:intrational-decrease-denominator}
		If $\beta_0,\beta_1,\beta_2 \geq 1$, one has
		\begin{align*}
			\mathcal{I}(\alpha,\beta) &= \tfrac 12 \big(
			\mathcal{I}(\alpha,\beta-\hat{v}_0) +
			\mathcal{I}(\alpha,\beta-\hat{v}_1) +
			\mathcal{I}(\alpha,\beta-\hat{v}_2)\big).
		\end{align*}
		\item 
		 \label{itm:intrational-decrease-alpha0}
			If $\beta_0 = 0$ and $\alpha_0 \geq 1$, and let $\alpha_j+\beta_j < \abs{\alpha}+1$ for some $j=1,2$, then 
		\begin{align*}
			\mathcal{I}(\alpha,\beta) & = \mathcal{I}(\alpha-\hat{v}_0,\beta-\hat{v}_{3-j}) - \mathcal{I}(\alpha-\hat{v}_0+\hat{v}_j,\beta).
		\end{align*}
		\item \label{itm:intrational-reduce-beta-or-improve-alpha}
		If $\beta_0 = 0$, $\beta_1,\beta_2 \geq 1$ and $\alpha_0> 0$, one has
		\begin{align*}
			\mathcal{I}(\alpha,\beta)
			&= \tfrac 12 \big(
			\mathcal{I}(\alpha,\beta-\hat{v}_1) +
			\mathcal{I}(\alpha,\beta-\hat{v}_2) \big)-\mathcal{I}(\alpha-\hat{v}_0+\hat{v}_1+\hat{v}_2,\beta) \\
			&\quad  + \tfrac12 \big( 
			\mathcal{I}(\alpha-\hat{v}_0+\hat{v}_1,\beta-\hat{v}_1)+
			\mathcal{I}(\alpha-\hat{v}_0+\hat{v}_2,\beta-\hat{v}_2) \big).
		\end{align*}   
	\end{enumerate}
\end{proposition}
\begin{proof}
	Recall that $\mathcal{I}(\alpha,\beta)< \infty$ holds if and only if $\norm{\alpha + \beta}_{\infty} \leq \abs{\alpha} + 1$, see Lemma~\ref{lem:regularity}. 
The proof of~\ref{itm:intrational-symmetry} follows by symmetry.

To prove~\ref{itm:intrational-decrease-denominator} note that with $\beta_0,\beta_1,\beta_2 \geq 1$ it follows that for any $j \in \{0,1,2\}$ one has $\norm{\alpha + \beta - \hat v_j}_{\infty} \leq  	\norm{\alpha + \beta}_{\infty} \leq \abs{\alpha} + 1.$
Consequently, all terms are finite. 
Then the claim follows from Lemma~\ref{lem:rat-rec}~\ref{itm:rec-4term-a}. 

By symmetry, we assume that the index satisfying the assumptions of~\ref{itm:intrational-decrease-alpha0} is $j=2$, that is, $\alpha_2 + \beta_2 < \abs{\alpha}+1$. 
We have $(\alpha-\hat{v}_0+\hat{v}_2)_2 + \beta_2 = \alpha_2 +\beta_2 + 1 \leq \abs{\alpha}+1$. 
Thus, Lemma~\ref{lem:regularity} shows that the integrals $\mathcal{I}(\alpha-\hat{v}_0,\beta-\hat{v}_1)$ and $ \mathcal{I}(\alpha-\hat{v}_0+\hat{v}_2,\beta)$ are finite.
Then the identity follows from Lemma~\ref{lem:rat-rec}~\ref{itm:rec-3term}. 
 
Similarly, under the conditions on $\alpha, \beta$ as specified in~\ref{itm:intrational-reduce-beta-or-improve-alpha} one can show that all terms are finite and the identity follows from Lemma~\ref{lem:rat-rec}~\ref{itm:rec-6term}. 
\end{proof}

\begin{figure}[h!]
	\begin{algorithm}[H]
		\caption{Computation of $\mathcal{I}(\alpha,\beta)$}
		\label{algo:ComputeI}
		\SetAlgoLined
		\SetKwFunction{KwComputeI}{ComputeI}
		\KwComputeI$(\alpha,\beta)$
		
		\KwIn{Multi-indices $\alpha,\beta \in \setN_0^3$ with $\alpha=(\alpha_0,\alpha_1,\alpha_2)$ and $\beta=(\beta_0,\beta_1,\beta_2)$}
		\KwOut{Integral mean $\mathcal{I}(\alpha,\beta)$}
		\If{ $\max\set{\alpha_0+\beta_0,\alpha_1+\beta_1,\alpha_2+\beta_2} > \alpha_0+\alpha_1+\alpha_2+1$} {
			\Return{$\infty$}
		}
		Sort $((\alpha_j,\beta_j))_{j=1,2}$ such that $\beta_0 \leq \beta_1 \leq \beta_2$
		
		\If{$\beta_0=\beta_1 = 0$}{
			\Return{$2\frac{\alpha_0! \alpha_1!\alpha_2!}{(\abs{\alpha}-\beta_2+2)!} \frac{(\alpha_0+\alpha_1+1-\beta_2)!}{(\alpha_0+\alpha_1+1)!}$}
		}
		\If(\tcp*[f]{$\beta \geq (1,1,1)$}){$\beta_0 \geq 1$}{
			\Return{$\tfrac{1}{2} \sum_{j=0}^2 \KwComputeI(\alpha,\beta - \hat{v}_j)$}
		}
		
			\If(\tcp*[f]{$\beta_0 = 0, \beta_2\geq \beta_1 \geq 1$}){$\alpha_0 =0$}{
				\Return{$\KwComputeJ(\alpha_1,\alpha_2,\beta_1,\beta_2)$}	
			}
		
		\If(\tcp*[f]{$\beta_0 = 0, \beta_2\geq \beta_1 \geq 1$, $\alpha_0\geq 1$}){$\alpha_1 + \beta_1 < |\alpha| + 1$}
		{    
			\Return{$\mathtt{ComputeI}(\alpha-\hat{v}_0,\beta-\hat{v}_2) - \mathtt{ComputeI}(\alpha-\hat{v}_0+\hat{v}_1,\beta)$}    
		}
		\If{$\alpha_2 + \beta_2 < |\alpha| + 1$}
		{
			\Return{$\mathtt{ComputeI}(\alpha-\hat{v}_0,\beta-\hat{v}_1) - \mathtt{ComputeI}(\alpha-\hat{v}_0+\hat{v}_2,\beta)$}    
		}
		\Else
		{
			\Return{$\tfrac{1}{2} \sum_{j=1}^2 \big( \mathtt{ComputeI}(\alpha,\beta - \hat{v}_j) + \mathtt{ComputeI}(\alpha-\hat{v}_0+\hat{v}_j,\beta-\hat{v}_j)\big)$ \\ $
				\qquad \qquad -\, \mathtt{ComputeI}(\alpha-\hat{v}_0+\hat{v}_1+\hat{v}_2,\beta)$  }
		}
	
	\end{algorithm}
\end{figure}

\begin{theorem}[Computation of~$\mathcal{I}$ in Algorithm~\ref{algo:ComputeI}]%
	\label{thm:Algo}
	Let $\alpha,\beta \in \setN^3_0$. 
	Then Algorithm~\ref{algo:ComputeI} terminates and returns~$\mathcal{I}(\alpha,\beta)$.
\end{theorem}
\begin{proof}
	The first condition determines whether the integral is finite as described in Lemma~\ref{lem:regularity}. 
	The sorting of the $(\alpha_j,\beta_j)$ guarantees that~$\beta_0 \leq \beta_1 \leq \beta_2$ using Proposition~\ref{pro:intrational-new}~\ref{itm:intrational-symmetry}. 
	In case $\beta_0 = \beta_1  = 0$ we use the explicit formula as in Lemma~\ref{lem:QuadRat-1}. 
	
	 If $\beta \geq (1,1,1)$, the routine exploits Proposition~\ref{pro:intrational-new}~\ref{itm:intrational-decrease-denominator} to reduce some~$\beta_j$. 
	 After that the remaining cases have $\beta_0=0$ and $1 \leq \beta_1 \leq \beta_2$. 
	 If also $\alpha_0=0$, then $\mathcal{I}(\alpha,\beta) = \mathcal{J}(\alpha_1,\alpha_2,\beta_1,\beta_2)$ and we can use Algorithm~\ref{algo:ComputeJ} to compute its value. 
	In the remaining cases we have additionally to the conditions on $\beta$ that $\alpha_0 \geq 1$.
	If either $\alpha_1 + \beta_1 < \abs{\alpha} + 1$ or $\alpha_2 + \beta_2 < \abs{\alpha} + 1$, then the three term recursion in Proposition~\ref{pro:intrational-new}~\ref{itm:intrational-decrease-alpha0} can be applied to reduce $\alpha_0$, and on one term also one of the entries of $\beta$. 
	Now, the only remaining case is $\beta_0 = 0$, $1 \leq \beta_1 \leq \beta_2$, $\alpha_0 \geq 0$, and $\alpha_1 + \beta_1 = \alpha_2 + \beta_2 = \abs{\alpha}+ 1$. 
	In this case Proposition~\ref{pro:intrational-new}~\ref{itm:intrational-reduce-beta-or-improve-alpha} is applicable, and in each term  $\alpha_0$ or $\abs{\beta}$ is reduced. 
	Since $\alpha_0$ and $\abs{\beta}$ are non-increasing in the iteration, and in each step $\alpha_0$ or $\abs{\beta}$ is reduced, the algorithm terminates after finitely many iterations. 
\end{proof}

\begin{remark}[Performance]
	\label{rem:many-calls}%
	Algorithm~\ref{algo:ComputeI} is recursive and is not optimized in terms of performance. 
	It can be improved by storing the values of $\mathcal{I}(\alpha,\beta)$ that are computed in the iteration and using symmetries. 
	Note that in the implementation of finite elements as in Section~\ref{sec:implementation},  Algorithm~\ref{algo:ComputeI} is solely used in the offline phase. 
\end{remark}

\section{Implementation}\label{sec:implementation}
In this section we describe our implementation of the finite elements discussed in Section~\ref{sec:FEspace} with \Matlab~\cite{Matlab.2024}, which is available at~\cite{Code}. Our code uses the implementation of a uniform mesh refinement routine from \cite{Bartels15}.
The implementation can easily be transferred to other programming languages. 
We discuss general finite elements using rational bubble functions in Section~\ref{sec:impl-general}. 
Then we address the special cases of the singular \Zienkiewicz element in Section~\ref{sec:implZienk} and the lowest-order \Guzman{}--Neilan element in Section~\ref{sec:implGN}.

Before we discuss our implementation, we explain the computation of the Jacobian $D\lambda = (\nabla \lambda)^\top$ needed for the computation of gradients in~\eqref{eq:trafogradx} and of Hessians in~\eqref{eq:chainrule-2}. 
Let $T = [v_0,v_1,v_2]$ be a simplex spanned by vertices $v_0,v_1,v_2 \in \setR^2$ and let $T^\textup{ref} = [v^\textup{ref}_0, v^\textup{ref}_1, v^\textup{ref}_2]$ denote the reference simplex with vertices $v^\textup{ref}_0 = (0,0)^\top$, $v^\textup{ref}_1= (1,0)^\top$, $v^\textup{ref}_2 = (0,1)^\top$ in $\mathbb{R}^2$. 
The affine transformation mapping $v^\textup{ref}_i \mapsto v_i$ for all $i= 0,1,2$ is denoted by $F \colon T^\textup{ref} \to T$. 
Its Jacobian matrix is given by 
\begin{equation*}
	D F = \begin{pmatrix}
		v_1-v_0,&v_2-v_0
	\end{pmatrix} = (\nabla F)^\top\in \mathbb{R}^{2\times 2}.
\end{equation*}
The barycentric coordinates of $T^\textup{ref}$ for any $x = (x_1,x_2)^\top \in \setR^2$ are given by 
\begin{align}\label{def:lambdaref}
	\lambda^\textup{ref}(x) = \begin{pmatrix}
		{\lambda}^\textup{ref}_0(x) \\
		{\lambda}^\textup{ref}_1(x)\\
		{\lambda}^\textup{ref}_2(x)	
	\end{pmatrix}
	= \begin{pmatrix}
		1 - x_1 - x_2\\
		x_1 \\ 
		x_2
	\end{pmatrix}.
\end{align}
The barycentric coordinates $\lambda = (\lambda_0,\lambda_1,\lambda_2)^\top$ on $T$ satisfy $\lambda \circ F = \lambda^\textup{ref}$. 
This identity and the chain rule lead with Jacobian matrices $D \lambda^\textup{ref}, D \lambda \in \setR^{3\times 2}$ to 
\begin{equation*}
 D \lambda^\textup{ref} = D (\lambda\circ F) = D \lambda\,  D F.
\end{equation*}
Hence, we can compute, as for example in Figure~\ref{fig:MatlabRoutineZ}, line~\ref{line:grad_lam}, the matrix
\begin{equation}\label{eq:nablaLam}
	G_T \coloneqq D \lambda =  D \lambda^\textup{ref}\, (D F)^{-\top} = \begin{pmatrix}
		-1 & -1\\
		1 & 0\\
		0 & 1 
	\end{pmatrix} 
	(D F)^{-1} \in \setR^{3 \times 2}.
\end{equation}
\begin{remark}[Alternative]
Lemma 3.10 in~\cite{Bartels15} states the alternative formula
  \begin{equation*}
    D \lambda = 
    \begin{pmatrix}
      1 & 1 & 1\\
      v_0 & v_1 & v_2 
    \end{pmatrix}^{-1} \begin{pmatrix}
      0 & 0 \\ 1 & 0 \\ 0&1 
    \end{pmatrix}.
 \end{equation*}
\end{remark}
\subsection{Finite elements using rational functions}\label{sec:impl-general}
This subsection provides a general framework to implement finite element schemes that involve rational functions of the form~\eqref{def:rat-fct}.

\subsubsection{The class `Rational Function'}\label{subsec:ClassRationalFunction}
	The local bases of the finite elements under consideration are linear combinations of rational functions. 
	That is with scalars $\gamma^{(j)} \in \mathbb{R}$ and multi-indices $\alpha^{(j)},\beta^{(j)} \in \setN_0^3$ for $j = 1,\dots,m \in \setN$ expressed in barycentric coordinates they have the form
\begin{equation}\label{eq:RatFctClass}
  \hat{b} = \sum_{j=1}^m \gamma^{(j)} \hat{\frR}^{\alpha^{(j)}}_{\beta^{(j)}}.
\end{equation}
Given a triangle~$T$ with barycentric coordinates $\lambda\colon T \to \widehat{T}$ we obtain the rational function $b = \hat{b} \circ \lambda$ in standard coordinates. 
In \Matlab we store such functions as tensors, that is, we identify $b$ with $\Matb \in \mathbb{R}^{3\times 3\times m}$, via
\begin{equation*}
  b \sim \Matb \in \mathbb{R}^{3\times 3\times m}\quad\text{with}\quad\Matb[:,:,j] =
  \begin{pmatrix}
    \alpha^{(j)}_0& \alpha^{(j)}_1 & \alpha^{(j)}_2\\
    \beta^{(j)}_0& \beta^{(j)}_1 & \beta^{(j)}_2\\
    \gamma^{(j)}& 0 & 0
  \end{pmatrix}\ \text{for }j=1,\dots,m.
\end{equation*}
Based on the results in Section~\ref{sec:rat-poly} we  implement several routines for rational functions $b \sim \Matb \in \mathbb{R}^{3\times 3\times m}$ and $c \sim \Matc \in \mathbb{R}^{3\times 3\times n}$, for $m,n \in \mathbb{N}$, including the following: 
\begin{itemize}
\item \lstinline|Multiply_Rationals(b,c)| returns the rational function $b\, c$ using~\eqref{eq:MultiplyRat},
\item \lstinline|Integrate_Rational(b)| returns $\dashint_T b \dx$ using Algorithm~\ref{algo:ComputeI},
\item \lstinline|Diff_Rational_lam(b)| returns the gradient $\nabla_{\lambda} \hat{b}$ using~\eqref{eq:diffRat}, 
\item \lstinline|Diff_Rational_x(DF,nabla_b)| and \lstinline|Diff_Rational_y(DF,nabla_b)| with $\mathtt{nabla\_b} = \nabla_\lambda \hat{b}$ and $\mathtt{DF} = DF$ return the derivatives $\partial_{x_1} b$ and $\partial_{x_2} b$ using~\eqref{eq:trafogradx}, \eqref{eq:nablaLam}.
\end{itemize}

\subsubsection{Compute local system matrices}\label{subsec:ComputeLocSysMat}
Let $(b_\ell)_{\ell = 1}^L$ with $L\in \setN$ denote a set of rational basis functions of a local finite element space $\mathcal{P}(T)$ on a simplex $T\in \tria$. 
Our routine computes local system matrices such as the mass matrix $M$ and the stiffness matrix $A$, defined by
\begin{align*}
\begin{aligned}
M &= (M_{\ell,k})_{\ell,k=1}^L&&\text{with } M_{\ell,k} \coloneqq \int_T b_\ell\, b_k\dx,\\
A &= (A_{\ell,k})_{\ell,k=1}^L&&\text{with } A_{\ell,k} \coloneqq \int_T \nabla b_\ell\cdot \nabla b_k\dx. 
\end{aligned}
\end{align*}
We can compute these local matrices directly by the routines mentioned in Section~\ref{subsec:ClassRationalFunction}, but this approach appears to be rather slow. 
Instead, the integrals are best calculated on a reference simplex using the routines in Section~\ref{subsec:ClassRationalFunction} and then transformed to~$T$ as exemplified in Section~\ref{sec:implZienk} and~\ref{sec:implGN} below. 
Notice that the reference simplex is $\widehat{T} \subset \RR^3$ defined in~\eqref{eq:defvhat}, that is, we compute quantities with respect to the barycentric coordinates. 
This distinguishes our ansatz from typical finite element implementations as for example in~\cite[Chap.~2.8]{Braess07}, where the reference element is $T^\textup{ref} \subset \RR^2$. 

\subsubsection{Evaluation of the right-hand side}\label{subsubsec:Rhs}
Let $f\colon \Omega \to \mathbb{R}$ be a function usually representing the right-hand side of a partial differential equation, and let $T \in \tria$ be a simplex.
Let $(\phi_j)_{j=1}^J \subset \mathcal{P}_r(T)$ denote a basis of the polynomial space $\mathcal{P}_r(T)$ for some $r\in \setN_0$ and recall the local basis functions $(b_\ell)_{\ell=1}^L$ mentioned in Section~\ref{subsec:ComputeLocSysMat}.  
Let~$\phi_j = \hat{\phi}_j \circ \lambda$ and $b_\ell = \hat{b}_\ell \circ \lambda$ and we define the affinely transformed basis on the reference triangle by $\phi^\textup{ref}_j = \hat{\phi}_j \circ \lambda^\textup{ref}$ and $b^\textup{ref}_\ell = \hat{b}_\ell \circ \lambda^\textup{ref}$  for all $j=1,\dots,J$ and $\ell = 1,\dots,L$, with $\lambda^{\textup{ref}}$ as in \eqref{def:lambdaref}. 
As described in the Section~\ref{subsec:ComputeLocSysMat} we can compute the matrix 
\begin{equation*}
\widehat{C} \in \mathbb{R}^{L\times J} \quad\text{with entries}\quad \widehat{C}_{\ell,j} = \dashint_{T^\textup{ref}} \phi^\textup{ref}_j\, b^\textup{ref}_\ell\dx = \dashint_T \phi_j b_\ell \dx.  
\end{equation*}
Suppose we have an approximation $\mathcal{I} f = \sum_{j=1}^J f_j \phi_j$ of $f$ with coefficient vector $\mathtt{f} = (f_j)_{j=1}^J\in \mathbb{R}^J$, obtained for example by means of nodal interpolation. 
Then for all $\ell = 1,\dots,L$, one has 
\begin{align*}
\dashint_T f\, b_\ell \dx \approx \dashint_T \mathcal{I} f\,  b_\ell \dx = \sum_{j=1}^J f_j \dashint_T \phi_j \, b_\ell\dx = (\widehat{C}\,  \mathtt{f})_\ell.
\end{align*}
In particular, the matrix vector multiplication $\widehat{C}\,  \mathtt{f}$ allows us to evaluation the right-hand side in our finite element scheme on each simplex $T\in \tria$. 
In our implementations we precompute $\widehat{C}^\top \eqqcolon \mathtt{bhat}$ and obtain $\mathtt{f}$ by nodal interpolation, cf.~line~\ref{line:f_loc1} in Figure~\ref{fig:MatlabRoutineZ}. 

\subsubsection{Change of basis}\label{sec:impl-gen-basischange}
Rather than using the basis $(b_\ell)_{\ell = 1}^L$ from the previous subsections, we want to use a basis $(c_\ell)_{\ell= 1}^L \subset \mathcal{P}(T) \coloneqq \linearspan\lbrace b_\ell\colon \ell = 1,\dots,L\rbrace$, that corresponds to the degrees of freedom $(\psi_\ell)_{\ell = 1}^L \subset \mathcal{P}(T)^*$ of the finite element in the sense that
\begin{align*}
\psi_\ell(c_m) = \delta_{\ell,m}\qquad\text{for all }\ell,m = 1,\dots, L.
\end{align*}
The representation of the basis functions $(c_\ell)_{\ell=1}^L$ might be 
independent of the underlying triangle $T\in \tria$, as for example in the case of Lagrange finite elements or, more generally, for affinely equivalent elements, see~\cite[Def.~3.4.1]{BrennerScott08}. 
However, since the elements discussed in this paper are not affinely equivalent, we compute the representation of the shape functions on the fly. 
For this purpose we define the generalized Vandermonde matrix
\begin{equation*}
\Vander = (\Vander_{\ell,k})_{\ell,k=1}^L \in \mathbb{R}^{L  \times L}\quad\text{with}\quad \Vander_{\ell,k} \coloneqq \psi_\ell(b_k)\quad\text{for all }\ell,k=1,\dots,L.
\end{equation*}
Then, the coefficients $y = (y_\ell)_{\ell=1}^L\in \mathbb{R}^L$ in the expansion $c_m = \sum_{\ell=1}^L y_\ell b_\ell$, for each $m = 1, \dots L$, are determined by
\begin{equation*}
\Vander y = e_m, 
\end{equation*}
where $e_m$ denotes the $m$-th canonical unit vector in $\mathbb{R}^L$. 
In other words, the coefficients of each basis function are the columns of the inverse matrix $\Vander^{-1}$.
We obtain the system matrices with respect to the basis $(c_\ell)_{\ell=1}^L$ by multiplying the matrices computed in Section~\ref{subsec:ComputeLocSysMat} by the matrix $\Vander^{-1}$ and its transpose from the left and right, e.g., for the mass matrix 
\begin{equation*}
\widetilde{M} = (\widetilde{M}_{\ell,k})_{\ell,k=1}^L \coloneqq \Vander^{-\top} M \Vander^{-1}\;\;\text{satisfies } \;\;\widetilde{M}_{\ell,k} = \int_T c_\ell \, c_k \dx\text{ for }\ell,k=1,\dots,L.
\end{equation*}
Similarly, we modify the computation of the right-hand side in Section~\ref{subsubsec:Rhs} by multiplying the matrix $\widehat{C}$ by $\Vander^{-\top}$, i.e., we replace $\widehat{C} \mathtt{f}$ by $\Vander^{-\top} \widehat{C} \mathtt{f}$.
\subsubsection{Assembly of the global system}\label{subsubsec:AssGlobSys}
With the local integrals computed as in the previous subsections, one can assemble the global system matrices and right-hand side. 
With global basis functions $(B_j)_{j=1}^N$ and $N\coloneqq \dim V_h$ and with $a(\bigcdot,\bigcdot) \colon V_h \times V_h \to \setR$ representing the bilinear form of the corresponding system matrix one has 
\begin{equation*}
\mathtt{A} = (\mathtt{A}_{j,k})_{j,k = 1}^N \in \mathbb{R}^{N \times N}\qquad\text{with } \mathtt{A}_{j,k} = a(B_j, B_k). 
\end{equation*}
The basis $(B_j)_{j=1}^N$ is chosen such that $\Psi_k(B_j) = \delta_{k,j}$ for all $j,k= 1,\dots, N$ with global degrees of freedom $\Psi_k$. 
We initialize $A$ as a $N\times N$ matrix with zero entries. 
Then we loop over all simplices $T\in \tria$. 
For each simplex $T\in \tria$ we compute the corresponding local system matrices as described above and add the values for the local degrees of freedom to the corresponding ones in $A$. 
Similarly, we proceed with the right-hand side, which leads to the vector $\mathtt{b} = (\mathtt{b}_j)_{j=1}^N \in \mathbb{R}^N$ with entries $\mathtt{b}_j = \int_\Omega \mathcal{I} f B_j\dx$ for all $j=1,\dots,N$.

\begin{remark}[Improvement]
In Figure~\ref{fig:MatlabRoutineZ} and~\ref{fig:MatlabGN} we use a slow assembling routine that updates the sparse Matrix $\mathtt{A}$ in each loop for the sake of a simpler presentation. 
For a more efficient implementation it is recommended to build the system matrix directly from an array of local system matrices and suitable index sets as described for example in~\cite{FunkenPraetoriusWissgott11}.
\end{remark}
%
\subsection{Singular \Zienkiewicz element}\label{sec:implZienk}

The previous section contains an approach to implement general finite elements that use polynomial and rational basis functions with degrees of freedom including point evaluations, evaluations of gradients, and normal derivatives on edges. 
However, the resulting routines are in general rather slow. 
We address this drawback by using the routines discussed in Section~\ref{subsec:ClassRationalFunction} to precompute certain quantities that are independent of the specific triangle $T \in \tria$.
This subsection discusses the resulting implementation for the singular \Zienkiewicz element introduced in Section~\ref{sec:Zienkiewicz} for the biharmonic problem:  One seeks $u_h \in V_h \cap H^2_0(\Omega)$ with $V_h \coloneqq \Zsing$  defined in Lemma~\ref{lem:sing-Zien} such that 
\begin{equation}\label{eq:biharmonic}
\int_\Omega \Delta u_h\, \Delta v_h \dx = \int_\Omega f v_h\dx\qquad\text{for all }v_h \in V_h \cap H^2_0(\Omega).
\end{equation}
Let us recall the local basis of $\Zsing$ on a simplex $T=[v_0,v_1,v_2]\in \tria$ with faces $\edges(T)= \{f_0,f_1,f_2\}$. 
The local basis, see~\eqref{eq:DefZienkiwiecz}, contains the basis of  $\mathcal{P}_2(T)$
\begin{equation*}
	\begin{aligned}
		b_1 \coloneqq \lambda_2^2, &&  b_2 \coloneqq  \lambda_1\lambda_2, && b_3 \coloneqq  \lambda_1^2, && b_4 \coloneqq  \lambda_0\lambda_2, && b_5 \coloneqq  \lambda_0\lambda_1, && b_6 \coloneqq  \lambda_0^2.
	\end{aligned}
\end{equation*}
The remaining basis functions of the reduced cubic Hermite element are
\begin{equation*}
\begin{aligned}
b_7 \coloneqq  \lambda_0^2 \lambda_1 - \lambda_0 \lambda^2_1, && b_8 \coloneqq   \lambda_1^2 \lambda_2 - \lambda_1 \lambda_2^2, && b_9 \coloneqq  \lambda_2^2\lambda_0 - \lambda_2\lambda^2_0,
\end{aligned}
\end{equation*}
and the rational bubble basis functions 
\begin{equation*}
\begin{aligned}
b_{10} \coloneqq  B_{f_0}, && b_{11} \coloneqq  B_{f_1}, && b_{12} \coloneqq  B_{f_2}.
\end{aligned}
\end{equation*}
The local degrees of freedom are 
\begin{equation}\label{eq:dofsZienkewicz}
\begin{aligned}
\psi_{1+j}(p) &= p(v_{j}),\   
\psi_{3+j}(p) = \partial_x p(v_{j}),  \ 
 \psi_{7+j}(p) = \partial_y p(v_{j}), \\
  \psi_{10+j}(p) &= \nabla p(\textup{mid}(\edge_{j})) \cdot \nu_{\edge_{j}},
\end{aligned}
\end{equation}
for all $p \in \Zsing(T)$ and $j=0,1,2$. 

In our implementation the vertices $\{v_1,\dots,v_m\} = \mathcal{N}$ and edges $\lbrace f_1,\dots,f_n\rbrace = \edges$ in the underlying triangulation $\tria$ are sorted according to their occurrence in the matrix $\mathtt{c4n} \in \mathbb{R}^{m\times 2}$ and $\mathtt{n4s} \in \mathbb{R}^{n\times 2}$. 
Those are computed in separate routines and contain the coordinates of each vertex and the vertex numbers of each edge as rows, respectively. 
The global degrees of freedom are
\begin{equation}
\begin{aligned}
&\Psi_{j}(w_h) = w_h(v_j),\quad  \Psi_{m+j}(w_h) = \partial_x w_h(v_j), \quad  \Psi_{2m+j}(w_h) = \partial_y w_h(v_j),\\
&\Psi_{3m+k}(w_h) = \nabla w_h(\textup{mid}(\edge_{k})) \cdot \nu_{\edge_{k}},
\end{aligned}
\end{equation}
 for all $w_h \in V_h$, $j=1,\dots,m$, and $k=1,\dots,n$. 
Let $(B_r)_{r=1}^{N} \subset V_h$ with dimension $N = 3m+n$ denote the corresponding basis functions in the sense of Section~\ref{subsubsec:AssGlobSys}. 
Moreover, we define the system matrix $\mathtt{A} = (\mathtt{A}_{r,s})_{r,s =1}^N$ for the biharmonic problem~\eqref{eq:biharmonic} and the right-hand side $\mathtt{b} = (\mathtt{b}_r)_{r=1}^N \in \mathbb{R}^N$ by
\begin{equation}\label{eq:GlobalSysMatZien}
\mathtt{A}_{r,s} = \int_\Omega \Delta B_r \, \Delta B_s \dx\quad\text{and}\quad \mathtt{b}_r\approx \int_\Omega f B_r \dx \quad \text{ for all  } r,s=1,\dots,N.
\end{equation}
While the computation of the right-hand side $\mathtt{b}$ has been discussed in Section~\ref{subsubsec:Rhs}, the computation of the matrix $\mathtt{A}$ is explained in the remainder of this subsection.
\subsubsection{Local stiffness matrix}\label{sec:impl-zien-Aloc}
We start with the computation of the matrix $A_T \in \mathbb{R}^{12\times 12}$ that represents the (local) bilinear form. Its entries read 
\begin{align}\label{eq:bilin-T}
(A_{T})_{r,s} \coloneqq \int_T \Delta b_r\, \Delta b_s \dx\qquad\text{for }r,s=1,\dots,12. 
\end{align}
The identity in~\eqref{eq:trafogradx} yields for rational polynomials $b\colon T \to \mathbb{R}$ with $b = \hat{b}\circ \lambda$
and with the $i$-th row $\partial_{x_i} \lambda$ of $\nabla \lambda \in \mathbb{R}^{2\times 3}$ that
\begin{equation}\label{eq:Formulapartb}
\partial_{x_i} b  = (\partial_{x_i} \lambda)\, (\nabla_\lambda \hat{b}) \circ  \lambda.
\end{equation}
Applying this formula twice leads to    
$\partial_{x_i}^2 b = (\partial_{x_i} \lambda)\, (\nabla_\lambda^2 \hat{b}) \circ \lambda\,  (\partial_{x_i} \lambda)^\top$.
  In particular, we obtain the representation
\begin{align}\label{eq:chainrule-Lapl}
\Delta b  = \partial^2_{x_1} b + \partial^2_{x_2} b = \sum_{i = 1,2}(\partial_{x_i} \lambda)\, (\nabla_\lambda^2 \hat{b}) \circ \lambda\,  (\partial_{x_i} \lambda)^\top.
\end{align}
Let us rewrite the matrix $A_T$ in~\eqref{eq:bilin-T}. For $v,w\in \setR^3$ and $H\in \setR^{3 \times 3}$ the  Frobenius product satisfies $v^\top H v = H : (v \otimes v)$ and thus
 \begin{equation}\label{eq:proofTempasdsafda}
  	v^\top H v + w^\top H w = H : (v \otimes v + w \otimes w ) = H : \left( (v\mid w) \begin{pmatrix}v^\top\\ w^\top \end{pmatrix} \right). 
 \end{equation}
Recall the matrix $G_T \coloneqq D \lambda \in \setR^{3\times 2}$ computed in~\eqref{eq:nablaLam}. 
Using~\eqref{eq:proofTempasdsafda} we can rewrite~\eqref{eq:chainrule-Lapl} as $\Delta b = (\nabla_\lambda^2 \hat{b} \circ \lambda) : G_T G_T^\top$.
Hence, for indices $r,s=1,\dots,12$ the integrand in~\eqref{eq:bilin-T} is given by 
\begin{equation}\label{eq:Prop1}
 \Delta b_r \Delta b_s 
 =	\big((\nabla_\lambda^2 \hat{b}_r \circ \lambda) : G_T G_T^\top\big) \big((\nabla_\lambda^2 \hat{b}_s \circ \lambda) : G_T G_T^\top\big).
\end{equation}
Let $A,B,C \in \setR^{3 \times 3}$ be matrices and define the tensors $[AB]$ with entries $[AB]_{i,j,k,l} = A_{i,j}  B_{k,l}$ and  $[CC]$ with entries $[CC]_{i,j,k,l} = C_{i,j}  C_{k,l}$. 
One has the identity
\begin{equation}\label{eq:Prop2}
\begin{aligned}
	(A:C) (B:C) &= \Big(\sum_{i,j} A_{i,j} C_{i,j}\Big) \Big(\sum_{k,l} B_{k,l} C_{k,l}\Big) \\
	&= \sum_{i,j,k,l} A_{i,j}  B_{k,l} C_{i,j} C_{k,l} = [AB]: [CC].
\end{aligned}
\end{equation}
 We define the tensors $\widehat{A}\in \setR^{12 \times 12 \times 3 \times 3 \times 3 \times 3}$ and $Q_T\in \setR^{3 \times 3 \times 3 \times 3}$ with entries 
\begin{align*}
\widehat{A}(r,s,i,j,k,l) 
&\coloneqq \dashint_T \left((\partial_{\lambda_i} \partial_{\lambda_j} \hat{b}_r) (
\partial_{\lambda_k} \partial_{\lambda_l} \hat{b}_s)\right) \circ \lambda  \dx, \\
{Q}_T(i,j,k,l) 
&\coloneqq (G_T G_T^\top)_{i,j}   (G_T G_T^\top)_{k,l}.  
\end{align*}
Notice that $\widehat{A}$, which is denoted by $\mathtt{Ahat}$ in Figure~\ref{fig:MatlabRoutineZ}, is independent of $T$ and thus is precomputed in our routine. 
Combining~\eqref{eq:Prop1} and~\eqref{eq:Prop2} leads to the identity
\begin{align}
	(A_T)_{r,s} = \int_{T} \Delta b_r \Delta b_s \dx = \abs{T}\, \widehat{A}(r,s,:,:,:,:) : {Q}_T.
\end{align} 
The \Matlab{} routine in Figure~\ref{fig:MatlabRoutineZ} performs this computation in lines~\ref{line:assAstart}--\ref{line:assAend}.
\begin{figure}[ht] 
\begin{lstlisting}[language=Matlab,escapechar=|]
for elem = 1 : nrElems
  nodes = n4e(elem,:);   % nodes of triangle
  coords = c4n(nodes,:); % coordinates of the three nodes
  sides = s4e(elem,:);   % sides of triangle
  DF = [coords(2,:)-coords(1,:);coords(3,:)-coords(1,:)]';
  grad_lam = [-1,-1;1,0;0,1]/DF;  % = [-1,-1;1,0;0,1]*inv(DF)|\label{line:grad_lam}|
  area = abs(det(DF))/2;  % area of current simplex T
  %% Compute local system matrix A_T %% 
  Q_T = tensorprod(grad_lam*grad_lam',grad_lam*grad_lam'); |\label{line:assAstart}|
  A_T = area*tensorprod(Ahat,Q_T,[3 4 5 6],[1 2 3 4]); |\label{line:assAend}|
  %% Local basis evaluation at dofs %% 
  Tgv = tensorprod(grad_lam,That_gv,1,3);	|\label{line:Tgv}|
  Tge = squeeze(pagemtimes(reshape(normal4s(sides,:)',[1,2,3]),permute(tensorprod(grad_lam,That_ge,1,3),[1,3,2])))'; |\label{line:Tge}|
  V = [That_v;squeeze(Tgv(1,:,:));squeeze(Tgv(2,:,:));Tge];|\label{line:Psi}|
  Basis4elem(:,:,elem) = inv(V);							|\label{line:Psiinv}|
  %% Update global matrix A and rhs b %%
  l2g_Dof = [nodes,nrNodes+nodes,2*nrNodes+nodes,3*nrNodes+sides]; |\label{line:defDof}|
  A(l2g_Dof,l2g_Dof) = A(l2g_Dof,l2g_Dof) + Basis4elem(:,:,elem)'*A_T*Basis4elem(:,:,elem);                                    
  b(l2g_Dof) = b(l2g_Dof) + area*Basis4elem(:,:,elem)'*bhat'*f(BaryCoords*coords);|\label{line:f_loc1}|
end 
\end{lstlisting}
\caption{\Matlab loop over all $T\in \tria$ to obtain the matrix $\mathtt{A}$ and right-hand side $\mathtt{b}$ in~\eqref{eq:GlobalSysMatZien}.}\label{fig:MatlabRoutineZ}
\end{figure}
\subsubsection{Local basis evaluation at dofs}\label{sec:impl-Zien-locbas}
In lines~\ref{line:Tgv}--\ref{line:Psiinv} in Figure~\ref{fig:MatlabRoutineZ} our routine computes a basis $(c_j)_{j=1}^{12}\subset Z(T)$ 
such that with degrees of freedom $\psi_j$ as in~\eqref{eq:dofsZienkewicz} we have
\begin{equation*}
\psi_j(c_k) = \delta_{j,k}\qquad\text{for all }j,k=1,\dots,12.
\end{equation*}
As described in Section~\ref{sec:impl-gen-basischange} for this purpose we invert the generalized Vandermonde matrix $\Vander = (\Vander_{j,k})_{j,k=1}^{12}\in \mathbb{R}^{12 \times 12}$ with $\Vander_{j,k} \coloneqq \psi_j(b_k)$ for all $j,k=1,\dots,12$. 
Its entries are computed as follows.

The point evaluations $\psi_1,\psi_2,\psi_3$ are independent of the underlying triangle $T$ and are stored in the precomputed matrix 
\begin{equation*}
\mathtt{That\_v} = (\Vander_{j,k})_{j=1,2,3}^{k=1,\dots,12} \in \mathbb{R}^{3\times 12}.
\end{equation*}
The dofs $\psi_4,\dots,\psi_9$ correspond to evaluations of the gradients at the vertices $v_0,v_1,v_2$ of $T$. 
We define the precomputed tensor $\mathtt{That\_gv} \in \mathbb{R}^{3\times 12 \times 3}$ with 
\begin{equation*}
\mathtt{That\_gv}(i,j,:) = (\nabla_\lambda \hat{b}_j)\circ \lambda(v_{i-1})\qquad\text{for all }i=1,2,3\text{ and }j=1,\dots,12.
\end{equation*}
Using~\eqref{eq:Formulapartb} this tensor allows us to compute the tensor $\mathtt{Tgv} \in \mathbb{R}^{2\times 3 \times 12}$ with entries $\mathtt{Tgv}(:,i,j) = \nabla b_j(v_{i-1})$ for all $i=1,2,3$ and $j=1,\dots,12$, cf.~line~\ref{line:Tgv} in Figure~\ref{fig:MatlabRoutineZ}.

It remains to consider the evaluations $\psi_{10},\psi_{11},\psi_{12}$ of the normal derivatives at the midpoints of edges. 
The gradients at the face midpoints are computed similarly as $\mathtt{Tgv}$. 
We then multiply these gradients with precomputed normal vectors. This leads to the matrix $\mathtt{Tge} \in \mathbb{R}^{3\times 12}$ with entries $\mathtt{Tge}(i,j) = \nabla b_j(\textup{mid}(\edge_{i-1}))\cdot \nu_{\edge_{i-1}}$ for all $i=1,2,3$ and $j=1,\dots,12$, cf.~line~\ref{line:Tge} in Figure~\ref{fig:MatlabRoutineZ}.
Combining these matrices leads to the matrix $\mathtt{V} = \Vander \in \mathbb{R}^{12\times 12}$ in line~\ref{line:Psi} in Figure~\ref{fig:MatlabRoutineZ}.
\subsubsection{Reduced \Zienkiewicz element} \label{subsec:RedZienkiewicz}
As discussed in Section~\ref{sec:Zienkiewicz} it is possible to reduce the number of basis functions, leading to the reduced singular \Zienkiewicz element in~\eqref{eq:ReducedZienkiewicz}. 
To obtain this reduced element, we have to modify our local basis functions $b_1,\dots,b_9$ by subtracting suitable scaled rational bubble functions $b_{10},b_{11},b_{12}$ such that the resulting functions $b \in \Zsing(T)$ satisfy $\nabla b|_{\edge_j} \cdot \nu_j \in \mathcal{P}_1(\edge_j)$ for all $j = 0,1,2$. This property is satisfied for the quadratic basis functions $b_1,\dots, b_6 \in \mathcal{P}_2(T)$ without any modification. 
For $b\in \lbrace b_7,b_8,b_9\rbrace$ we use the ansatz 
\begin{equation*}
\bar{b} \coloneqq b - \sum_{j=0}^2 \gamma_j B_{\edge_j}\qquad\text{with coefficients }\gamma_0,\gamma_1,\gamma_2 \in \mathbb{R}\text{ to be determined}.
\end{equation*} 
According to Lemma~\ref{lem:RatBubbles} the function $\nabla \bar{b}|_{\edge_j}$ is a quadratic function on the edge $\edge_j = [v_{j+1},v_{j+2}]$, for $j =  0,1,2$. 
Hence, the function $\bar{b}$ satisfies $\nabla \bar{b}|_{\edge_j} \cdot \nu_j \in \mathcal{P}_1(\edge_j)$ if and only if 
\begin{equation*}
\nabla \bar{b}(\textup{mid}(\edge_j))\cdot \nu_{\edge_j}
 =\tfrac12 \left( \nabla \bar{b}(v_{j+1})
  + \nabla \bar{b}(v_{j+2})\right)\cdot \nu_{\edge_j}.
\end{equation*}
Due to the properties of the rational bubble functions stated in Lemma~\ref{lem:RatBubbles} and by the definition of $\bar{b}$ this is equivalent to
$$\nabla \left( b(\textup{mid}(\edge_j)) - \gamma_j B_{\edge_j}(\textup{mid}(\edge_j))\right)\cdot \nu_{\edge_j}
 = \tfrac12 \big( \nabla b(v_{j+1}) + \nabla b(v_{j+2})\big)\cdot \nu_{\edge_j}.$$ 
This determines the coefficients to be 
\begin{equation*}
\gamma_j = \frac{\nabla b(\textup{mid}(\edge_j))\cdot \nu_{\edge_j} -\tfrac12  \big( \nabla b(v_{j+1}) + \nabla b(v_{j+2})\big)\cdot \nu_{\edge_j}}{B_{\edge_j}(\textup{mid}(\edge_j))\cdot \nu_{\edge_j} }\qquad \text{for }j=0,1,2. 
\end{equation*}
The values on the right-hand side have been computed and stored in the matrix $\mathtt{V}\in \mathbb{R}^{12\times 12}$, see Section~\ref{sec:impl-Zien-locbas}, in the sense that for all $k=1,\dots,9$ and $j=0,1,2$
\begin{align*}
\nabla b_k(\textup{mid}(\edge_j))\cdot \nu_{\edge_j} &= \mathtt{V}(10+j,k),\\
\nabla b_k(v_j) &= (\mathtt{V}(4+j,k),\mathtt{V}(7+j,k))^\top,\\
B_{\edge_j}(\textup{mid}(\edge_j))\cdot \nu_{\edge_j} &= b_{10+j}(\textup{mid}(\edge_j))\cdot \nu_{\edge_j} = \mathtt{V}(10+j,10+j).
\end{align*}
This allows us to compute the values $\gamma_j$ as in line~\ref{line:rho0}--\ref{line:rho2} in Figure~\ref{fig:MatlabRoutineZreduced}. 
In this way we obtain the basis functions $\bar{b}_1,\dots,\bar{b}_9$ of $\Zred(T)$ with $\bar{b}_k = b_k$ for $k=1,\dots,6$. 
Since the degrees of freedom $\psi_1,\dots,\psi_9$ do not see the correction by the rational bubble functions according to Lemma~\ref{lem:RatBubbles}, that is, $\psi_k(B_{\edge_j}) = 0$ for $k=1,\dots,9$ and $j=0,1,2$, the Vandermonde matrix remains unchanged in the sense that $\psi_k(\bar{b}_\ell) = \psi_k(b_\ell)$ for all $k,\ell=1,\dots,9$. 
This allows us to compute the coefficients $(y_\ell)_{\ell=1}^{12} \in \mathbb{R}^{12}$ of the functions $\bar{c}_k = \sum_{\ell=1}^{12} y_\ell b_\ell \in \Zred(T)$ with $\psi_{\ell}(\bar{c}_k) = \delta_{\ell,k}$ for all $k,\ell=1,\dots,9$ as illustrated in lines~\ref{line:basisRed0}--\ref{line:basisRed} in Figure~\ref{fig:MatlabRoutineZreduced}.  

\begin{figure}
\begin{lstlisting}[language=Matlab,escapechar=|]
gamma0 = (V(10,7:9)-normal4s(sides(1),:)/2*(V([5,8],7:9)+V([6,9],7:9)))/V(10,10);|\label{line:rho0}|
gamma1 = (V(11,7:9)-normal4s(sides(2),:)/2*(V([6,9],7:9)+V([4,7],7:9)))/V(11,11);
gamma2 = (V(12,7:9)-normal4s(sides(3),:)/2*(V([4,7],7:9)+V([5,8],7:9)))/V(12,12);|\label{line:rho2}|
RedBasis = [eye(9);zeros(3,6),-[gamma0;gamma1;gamma2]];|\label{line:basisRed0}|
Basis4elem(:,:,elem) = RedBasis/V(1:9,1:9);|\label{line:basisRed}|
l2g_Dof = [nodes,nrNodes+nodes,2*nrNodes+nodes];
\end{lstlisting}
\caption{Modification replacing lines~\ref{line:Psiinv}--\ref{line:defDof} in Figure~\ref{fig:MatlabRoutineZ} to obtain the reduced singular \Zienkiewicz element.}\label{fig:MatlabRoutineZreduced}
\end{figure}

\subsection{\Guzman{}--Neilan element}\label{sec:implGN}

In this subsection we present an implementation of the 2D \Guzman{}--Neilan element for the Stokes problem.
More precisely, we consider the \Guzman{}--Neilan element as presented in Section~\ref{sec:GNelement} with global spaces $\Vspace$ and $\Qspace$ defined in Lemma~\ref{lem:GN} with given underlying triangulation $\tria$. The discretized Stokes problems seeks $(u_h,\pi_h) \in (\Vspace \cap H^1_0(\Omega)^2)\times \Qspace \cap L^2_0(\Omega)$ such that 
		\begin{align}\label{eq:StokesGN}
		\begin{aligned}
			\int_\Omega \nabla u_h : \nabla v_h \dx -
			 \int_{\Omega} \pi_h \divergence v_h \dx  &= \int_\Omega f v_h\dx&\quad&\text{for all }v_h \in \Vspace \cap H^1_0(\Omega)^2,\\
			\int_\Omega q_h \divergence  u_h &= 0 &\quad&\text{for all }q_h \in \Qspace  \cap L^2_0(\Omega).
		\end{aligned}
	\end{align}
To define our local basis functions for any simplex $T=[v_0,v_1,v_2]$, recall the functions $b_{f_i}$ and $B_{f_i}$ with $i = 0,1,2$ defined in~\eqref{eq:DefRatBubble}, and set $e_1 = (1,0)^\top$, $e_2 = (0,1)^\top$.
The local velocity space $V(T)$ is spanned by the basis functions
\begin{equation}\label{eq:DefBasisGN}
	\begin{aligned}
		b_1 &\coloneqq \lambda_0 e_1, &  
		b_2 &\coloneqq  \lambda_1 e_1, & 
		b_3 &\coloneqq  \lambda_2 e_1 , \\ 
		b_4 &\coloneqq \lambda_0 e_2, & 
		b_5 &\coloneqq  \lambda_1 e_2, & 
		b_6 &\coloneqq  \lambda_2 e_2, \\
		b_7 &\coloneqq \bfcurl (\lambda_0^2 \lambda_1 - \lambda_1^2 \lambda_0), &  
		b_8 &\coloneqq \bfcurl (\lambda_1^2 \lambda_2 - \lambda_2^2 \lambda_1), &  
	    b_9 &\coloneqq \bfcurl (\lambda_2^2 \lambda_0 - \lambda_0^2 \lambda_2), \\
	    b_{10} &\coloneqq \bfcurl (B_{f_{0}}), &  
	    b_{11} &\coloneqq \bfcurl (B_{f_{1}}), &  
	    b_{12} &\coloneqq \bfcurl (B_{f_{2}}).
	\end{aligned}
\end{equation}
Let us additionally define  
\begin{equation*}
	\begin{aligned}
		\rho_1 &\coloneqq \lambda_0^2 \lambda_1 - \lambda_1^2 \lambda_0,\quad &  
		\rho_2 &\coloneqq  \lambda_1^2 \lambda_2 - \lambda_2^2 \lambda_1, \quad&  
		\rho_3 &\coloneqq \lambda_2^2 \lambda_0 - \lambda_0^2 \lambda_2, \\
		\rho_{4} &\coloneqq B_{f_{0}}, &  
		\rho_{5} &\coloneqq  B_{f_{1}}, &  
		\rho_{6} &\coloneqq B_{f_{2}}.
	\end{aligned}
\end{equation*}
The local basis of $Q(T)$ consists of the single constant function 
\begin{align*}
	q_1 = 1.
\end{align*}
The local degrees of freedom are for all $p \in V(T)$ and $j = 0,1, 2$ defined by
\begin{equation}\label{eq:dofsGuzmanNeilan}
\begin{aligned}
\psi_{1+j}(v) &= (p(v_j))_1,& \psi_{4+j}(p)& = (p(v_j))_2, \\ 
\psi_{7+j}(p) &= p(\textup{mid}(\edge_j)) \cdot \nu_{\edge_j},& \psi_{10+j}(p)& = p(\textup{mid}(\edge_j)) \cdot \tau_{\edge_j}.
\end{aligned}
\end{equation}
In our implementation the vertices $\{v_1,\dots,v_m\} = \mathcal{N}$ and edges $\lbrace f_1,\dots,f_n\rbrace = \edges$ in the underlying triangulation $\tria$ are sorted according to their occurrence in the matrix $\mathtt{c4n} \in \mathbb{R}^{m\times 2}$ and $\mathtt{n4s} \in \mathbb{R}^{n\times 2}$, which contains the coordinates of each vertex and the vertex numbers of each edge, respectively. Then the global degrees of freedom are, for all $w_h \in \Vspace$, $j=1,\dots,m$, and $k=1,\dots,n$,
\begin{equation}
\begin{aligned}
\Psi_{j}(w_h) &= (w_h(v_j))_1,& \Psi_{m+j}(w_h)& = (w_h(v_j))_2,\\
\Psi_{2m+k}(w_h) &= w_h(\textup{mid}(\edge_k)) \cdot \nu_{\edge_k},& \Psi_{2m+n+k}(w_h)& = w_h(\textup{mid}(\edge_k)) \cdot \tau_{\edge_k}.
\end{aligned}
\end{equation}
The corresponding global basis in the sense of Section~\ref{subsubsec:AssGlobSys} is denoted by $(B_r)_{r=1}^{N} \subset V_h$ with dimension $N = 2m+2n$. 
Moreover, in our implementation the simplices $\lbrace T_1,\dots,T_p\rbrace = \tria$ are sorted according to their occurrence in the matrix $\mathtt{n4e} \in \mathbb{R}^{p\times 3}$ whose $j$-th row contains the indices of the vertices of $T_j$. 
Then $(q_j)_{j=1}^p$ with $q_j \coloneqq \indicator_{T_j}$ for all $j=1,\dots,p$ is a basis of $\Qspace$. The remainder of this subsection focuses on the computation of the local contribution of the matrices $\mathtt{A} = (\mathtt{A}_{r,s})_{r,s =1}^N \in \mathbb{R}^{N\times N}$ and $\mathtt{B} = (\mathtt{B}_{r,j})_{r=1,\dots,N}^{j=1,\dots,p} \in \mathbb{R}^{N\times p}$ as well as of the right-hand side $\mathtt{b} = (\mathtt{b}_r)_{r=1}^N \in \mathbb{R}^N$ such that for all $r,s=1,\dots,N$ and $j=1,\dots,p$
\begin{equation}\label{eq:GlobalSysMatGN}
\mathtt{A}_{r,s} = \int_\Omega \nabla B_r : \nabla B_s \dx,\quad 
\mathtt{B}_{r,j} = \int_\Omega  q_j \divergence· B_r\dx,\quad\mathtt{b}_r\approx \int_\Omega f B_r \dx.
\end{equation}
\begin{figure}
\begin{lstlisting}[language=Matlab,escapechar=|]
for elem = 1 : nrElems
  nodes = n4e(elem,:);   % nodes of triangle
  coords = c4n(nodes,:); % coordinates of the three nodes
  sides = s4e(elem,:);   % sides of triangle
  DF = [coords(2,:)-coords(1,:);coords(3,:)-coords(1,:)]';
  grad_lam = [-1,-1;1,0;0,1]/DF;  %grad_lam = [-1,-1;1,0;0,1]*inv(DF)
  area = abs(det(DF))/2;  % area equals volume of current simplex
  %% Compute local stiffness matrix A_T (P_T, R_T ,M_T) %%
  GG = grad_lam*grad_lam';
  A_T(1:3,1:3) = area*GG;  A_T(4:6,4:6) = area*GG;  %P_T |\label{line:P_T}|
  Q_T = tensorprod(GG,GG);
  A_T(7:12,7:12) = area*tensorprod(Rhat,Q_T,[3 4 5 6],[1 2 3 4]); %R_T|\label{line:R_T}|
  W_T = tensorprod(R*grad_lam',GG);
  A_T(1:6,7:12) = area*tensorprod(Mhat,W_T,[3 4 5 6],[1 2 3 4]); %M_T |\label{line:M_T}|
  A_T(7:12,1:6) = A_T(1:6,7:12)';
  B_T = area*grad_lam(:); |\label{line:B_T}|
  %% Local basis evaluation at dofs %% 
  normals = normal4s(sides,:);   %normal vector for each side 
  tangents = tangent4s(sides,:); %tangent vector for each side 
  V_left = [eye(6);normals(1,:)*val_mid1';normals(2,:)*val_mid2';normals(3,:)*val_mid3';tangents(1,:)*val_mid1';tangents(2,:)*val_mid2';tangents(3,:)*val_mid3'];   |\label{line:Psi_left}| 
  Tgv = tensorprod(grad_lam,That_gv,1,3);  |\label{line:TgvGN}| 
  Tge = tensorprod(grad_lam,That_ge,1,3); |\label{line:TgvAuxGN}| 
  for i=1:3
    Tge_nt([i,i+3],:) = [normals(i,:);tangents(i,:)]*R*squeeze(Tge(:,i,:));  end|\label{line:Tge_nt}|
  V_right = [squeeze(Tgv(2,:,:));-squeeze(Tgv(1,:,:));Tge_nt];
  V = [V_left,V_right]; |\label{line:Vander}|
  Basis4elem(:,:,elem) = inv(V);|\label{line:invVanderGN}|
  %% Update global matrix A and rhs b %%
  l2g_Dof = [nodes,nrNodes+nodes,2*nrNodes+sides,2*nrNodes+nrSides+sides];|\label{line:invVanderGN2}|
  A(l2g_Dof,l2g_Dof) = A(l2g_Dof,l2g_Dof) + Basis4elem(:,:,elem)'*A_T*Basis4elem(:,:,elem);
  B(l2g_Dof,elem) = Basis4elem(1:6,:,elem)'*B_T;
  f_loc = f(BaryCoords*coords); |\label{line:f_loc2}|
  b2 = tensorprod(grad_lam,bhat2,1,2); |\label{line:b2}|
  b_T1 = [bhat1'*f_loc(:,1);bhat1'*f_loc(:,2)]; 
  b_T2 =squeeze(b2(2,:,:))'*f_loc(:,1)-squeeze(b2(1,:,:))'*f_loc(:,2);|\label{line:bT2}|
  b(l2g_Dof) = area*b(l2g_Dof)+Basis4elem(:,:,elem)'*[b_T1;b_T2]; |\label{line:f_loc2End}|                           
end 
\end{lstlisting}
\caption{Computation of the local contributions to assemble the matrices~$ \mathtt{A}$, $\mathtt{B}$, and $\mathtt{b}$ in~\eqref{eq:GlobalSysMatGN}.}\label{fig:MatlabGN}
\end{figure}%
\subsubsection{Local stiffness matrix}
First, we describe the computation of the local stiffness matrix $A_T \in \setR^{12 \times 12}$. Given a simplex $T=[v_0,v_1,v_2]\in \tria$ its entries read 
\begin{align*}
	(A_T)_{r,s}  = \int_T \nabla b_r : \nabla b_s \dx\qquad\text{for all }r,s=1,\dots,12.
\end{align*}
We divide this matrix into blocks $R_T, M_T, P_T \in \setR^{6 \times 6}$ in the sense that 
\begin{align*}
	A_T = \begin{pmatrix}
		P_T & M_T \\
		M_T^\top & R_T 
		\end{pmatrix}. 
\end{align*}
This means $P_T$ contains the entries where both basis functions are affine polynomials, $R_T$ contains the entries where both basis functions are curls of a rational functions, and $M_T$ contains the mixed ones. 

Let us start with computing $P_T$. 
Its entries are 
\begin{align*}
(P_T)_{r,s} \coloneqq  \int_T \nabla b_r : \nabla b_s \dx\qquad\text{for }r,s=1,\dots,6.
\end{align*}
For all $j,k=1,2,3$ we have the identity
\begin{equation*}
P_{j,k} = \int_T \nabla \lambda_{j-1} \cdot \nabla \lambda_{k-1} \dx = \int_T \nabla b_j : \nabla b_k\dx = \int_T \nabla b_{3+j} : \nabla b_{3+k}\dx.
\end{equation*}
Moreover, the off-diagonal blocks equal zero, that is, $P_{j,3+k} = 0 = P_{3+j,k}$ for all $j,k=1,2,3$.  
Hence, with the matrix $G_T = D \lambda\in \mathbb{R}^{3\times 2}$ computed as in~\eqref{eq:nablaLam} we have, cf.~Figure~\ref{fig:MatlabGN}, line~\ref{line:P_T}, that
\begin{equation*}
P_T =  \begin{pmatrix}
|T|\, G_T G_T^\top & 0 \\
0 & |T|\, G_T G_T^\top
\end{pmatrix}.
\end{equation*}

Let us now consider $R_T \in \setR^{6 \times 6}$. 
Recall that in 2D the curl is given by
\begin{align*}
\bfcurl f = R \nabla f \qquad \text{with rotation matrix } R = \begin{pmatrix} 
0 & 1 \\
-1 & 0	
\end{pmatrix}.
\end{align*}
Hence, we obtain for all $r,s=1,\dots,6$ that
\begin{align*}
	(R_T)_{r,s} 
	&\coloneqq \int_T \nabla b_{r+6} : \nabla b_{s+6} \dx = \int_T (\nabla \bfcurl\rho_r) : ( \nabla \bfcurl \rho_{s}) \dx\\
	& = \int_T (\nabla R \nabla \rho_r) : (\nabla R \nabla  \rho_{s}) \dx 
	= \int_T  \nabla^2 \rho_r : \nabla^2 \rho_{s} \dx.
\end{align*} 
With $G_T = D\lambda$ and~\eqref{eq:chainrule-2} the integrand equals 
\begin{align*}
	\nabla^2 \rho_r : \nabla^2 \rho_{s}
	= \left( G_T^\top (\nabla_{\lambda}^2 \hat{\rho}_r)\circ \lambda\,  G_T \right) : \left( G_T^\top (\nabla_{\lambda}^2 \hat{\rho}_s) \circ \lambda\,   G_T \right). 
\end{align*}
We define the tensors $\widehat{R} \in\setR^{6 \times 6 \times 3 \times 3 \times 3 \times 3}$ and ${Q}_T\in\setR^{ 3 \times 3 \times 3 \times 3} $ as
\begin{align*}
	\widehat{R}(r,s,i,j,k,l) 
	&\coloneqq \dashint_T \left((\partial_{\lambda_i} \partial_{\lambda_k} \hat{\rho}_r) (
	\partial_{\lambda_j} \partial_{\lambda_l} \hat{\rho}_s)\right)\circ \lambda  \dx, \\
	{Q}_T(i,j,k,l) 
	&\coloneqq (G_T G_T^\top)_{i,j}   (G_T G_T^\top)_{k,l}.
\end{align*}
Component-wise computation reveals the identity $(R_T)_{r,s} =  \abs{T}\, \widehat{R}(r,s,:,:,:,:) : {Q}_T$.  
In our implementation the element-independent tensor $\widehat{R} = \mathtt{Rhat}$ is precomputed and the computation of $R_T = \mathtt{R\_T}$ is shown in Figure~\ref{fig:MatlabGN}, line~\ref{line:R_T}.

It remains to compute the submatrix $M_T$ with entries, for $r,s = 1,\ldots, 6$, 
\begin{align*}
	(M_T)_{r,s} &=  \int_T \nabla b_r : \nabla b_{s+6}\dx  = 
	\int_T \nabla b_r : \nabla R \nabla \rho_{s} \dx 	= \int_T \nabla b_r : R \nabla^2 \rho_{s} \dx.
\end{align*}
With the identities~\eqref{eq:trafogradx}, \eqref{eq:chainrule-2}, and $G_T = D \lambda $ the integrand has the form 
\begin{align*}
 \nabla b_r  : R \nabla^2 \rho_{s} = (\nabla_\lambda \hat{b}_r) \circ \lambda\, G_T : R (G_T^\top (\nabla^2_{\lambda} \hat{\rho}_s)\circ \lambda \, \, G_T ).
\end{align*} 
Let us define the tensors $\widehat{M} \in\setR^{6 \times 6 \times 2 \times 3 \times 3 \times 3}$ and ${W}_T\in\setR^{ 2 \times 3 \times 3 \times 3} $ by
\begin{align*}
	\widehat{M}(r,s,i,j,k,l) 
	&\coloneqq \dashint_T \big(\partial_{\lambda_k}  (\hat{b}_r)_{i} (
	\partial_{\lambda_j} \partial_{\lambda_l} \hat{\rho}_s)\big)\circ \lambda \dx, \\
	{W}_T(i,j,k,l) 
	&\coloneqq (R G_T^\top)_{i,j}   (G_T G_T^\top)_{k,l}.
\end{align*}
This leads to the identity $(M_T)_{r,s} =\abs{T}\,  \widehat{M}(r,s,:,:,:,:) : {W}_T$.  
In our implementation the element-independent tensor $\widehat{M} = \mathtt{Mhat}$ is precomputed and the computation of $M_T = \mathtt{M\_T}$ is shown in Figure~\ref{fig:MatlabGN}, line~\ref{line:M_T}.

\subsubsection{System matrix}

Since the pressure space is one-dimensional, the contributions $\int_{T} q_1 \divergence b_r \dx$ with $q_1 = 1$ are represented by a vector $B_T \in \setR^{12}$ with entries 
\begin{align*}
	(B_T)_{r} = \int_T \divergence b_r(x) \dx\qquad\text{for }r = 1, \dots, 12. 
\end{align*}
It follows by $\divergence \bfcurl \rho = 0$ that $ 
(B_T)_{r} = 0$ for all $r = 7, \dots,12$. 
Due to the definition of the basis in~\eqref{eq:DefBasisGN} for $j=1,2,3$ we have that 
\begin{equation*}
\divergence b_j = \partial_x \lambda_{j-1}  \qquad\text{and}\qquad \divergence b_{3+j} = \partial_y \lambda_{j-1}.
\end{equation*}
In particular, we obtain
\begin{align*}
(B_T)_{r} =  
\begin{cases}
\abs{T}\, \partial_x \lambda_{r-1} & \text{for } r = 1, \dots, 3, \\
\abs{T}\, \partial_y \lambda_{r-4} & \text{for } r = 4, \dots, 6, \\
0&\text{for } r = 7, \dots, 12.
\end{cases}
\end{align*}
This leads to the computation of $(B_T)_{r=1}^6 = \mathtt{B\_T}$ as in line~\ref{line:B_T} in Figure~\ref{fig:MatlabGN}.
\subsubsection{Local basis evaluation at dofs}
Since point evaluations of the basis functions $b_1,\dots,b_6$ at vertices and midpoints of edges are independent of the underlying simplex $T = [v_0,v_1,v_2]$, we obtain the left side $\mathtt{V\_left} = (\Vander_{r,s})_{r=1,\dots,12}^{s=1,\dots,6} \in \mathbb{R}^{12\times 6}$ of the generalized Vandermonde matrix by multiplying precomputed point evaluations with the corresponding normal and tangential vectors, see Figure~\ref{fig:MatlabGN}, line~\ref{line:Psi_left}.

To obtain the remaining part $\mathtt{V\_right} = (\Vander_{r,s})_{r=1,\dots,12}^{s=7,\dots,12}$ of the Vandermonde matrix, we apply the considerations in  Section~\ref{sec:impl-Zien-locbas} to compute in line~\ref{line:TgvGN}--\ref{line:TgvAuxGN} in Figure~\ref{fig:MatlabGN} the tensors $\mathtt{Tgv},\mathtt{Tge}\in \mathbb{R}^{2\times3\times 6}$ with entries  
\begin{equation*}
   \mathtt{Tgv}(:,j,s) = \nabla \rho_s(v_{j-1})\quad\text{and}\quad  \mathtt{Tge}(:,j,s)= \nabla \rho_s(\textup{mid}(\edge_{j-1})),
\end{equation*}
for all $j=1,2,3$  and $s=1,\dots,6$. 
While the values in $\mathtt{Tgv}$ can be directly added to the Vandermonde matrix $V$, one has to rotate and multiply the values in $\mathtt{Tge}$ by the normal and tangential vectors as done in line~\ref{line:Tge_nt} in Figure~\ref{fig:MatlabGN}. 
Combining the values leads to the matrix $\Vander = \mathtt{V}$ in line~\ref{line:Vander} in Figure~\ref{fig:MatlabGN}.

\subsubsection{Assembly of the right-hand side}
Let $\mathcal{P}_m(T)$ be the polynomial space of maximal degree $m\in \mathbb{N}_0$ with Lagrange basis $\{\varphi_1,\dots,\varphi_J\}$.
In lines~\ref{line:f_loc2}--\ref{line:f_loc2End} of Figure~\ref{fig:MatlabGN} the right-hand side is assembled. 
The evaluation of the local components $\int_T\mathcal{I} f b_j \dx$ for $j=1,\dots,6$ with nodal interpolation operator $\mathcal{I}\colon C^0(T) \to \mathcal{P}_m(T)$ is described in Section~\ref{subsubsec:Rhs}. To evaluate the local components  $\int_T \mathcal{I} f \bfcurl \rho_j \dx$ for $j=1,\dots,6$, the code precomputes the tensor $\mathtt{bhat2}\in \mathbb{R}^{J \times 3 \times 6}$ containing the values
\begin{equation*}
\mathtt{bhat2}(j,k,r) = \dashint_T (\hat{\phi}_j \partial_{\lambda_k} \hat{\rho}_r) \circ \lambda\dx \quad \text{for all } j=1,\dots,J, \ k=1,\dots,3,\ r=1,\dots,6.
\end{equation*}
Exploiting the representation of $\nabla \rho_s = \nabla \lambda (\nabla_\lambda \hat{\rho}_s) \circ \lambda$ in~\eqref{eq:trafogradx}, we obtain in line~\ref{line:b2} in Figure~\ref{fig:MatlabGN} the tensor $\mathtt{b2}$ with entries 
\begin{equation*}
\mathtt{b2}(m,j,r) = \dashint_T \phi_j \partial_{x_m} \rho_r \dx\quad\text{for all }m = 1,2,\ j=1,\dots,J,\ r=1,\dots,6. 
\end{equation*}
With this we evaluate in line~\ref{line:bT2} in Figure~\ref{fig:MatlabGN} the contributions
\begin{equation*}
\mathtt{b\_T2}(r) = \dashint_T \mathcal{I} f\,  \bfcurl \rho_r \dx = \dashint_T \mathcal{I}f\,  b_{6+r} \dx\qquad\text{for all }r = 1,\dots,6.
\end{equation*}
These contributions are added to obtain the global vector $\mathtt{b}$.
\subsubsection{Reduced \Guzman{}--Neilan element}
To obtain the reduced \Guzman--Neilan element described in Lemma~\ref{lem:GNred}, we have to correct the local basis $b_1,\dots,b_9$ by the rational bubble functions $b_{10} = \bfcurl B_{\edge_0}$, $b_{11} = \bfcurl B_{\edge_1}$, $b_{12}=\bfcurl B_{\edge_2}$ such that the resulting functions $\bar{b} \in \Vred(T)$ satisfy
\begin{equation}\label{eq:CondGNref}
\bar{b}|_{\edge_j} \cdot \tau_{\edge_j} \in \mathcal{P}_1(f_j)\qquad\text{for all }j=0,1,2.
\end{equation}
This is trivially satisfied for the affine basis functions $b_1 = \bar{b}_1,\dots,b_6 = \bar{b}_6 \in \mathcal{P}_1(T)^2$. For the corrections $\bar{b}_k = b_k - \sum_{\ell=1}^3 \gamma_\ell b_{9+\ell}$ with $k=7,8,9$ and coefficients $\gamma_1,\gamma_2,\gamma_3\in \mathbb{R}$ of the remaining basis functions 
we exploit the identity
\begin{align*}
 \bar{b}_k|_{\edge_j} \cdot \tau_{\edge_j} & = \bfcurl \Big( \rho_{k-6} - \sum_{\ell=1}^3 \gamma_\ell \rho_{3+\ell} \Big)\Big|_{\edge_j} \cdot \tau_{\edge_j}  = \bfcurl \big( \rho_{k-6} - \gamma_j B_{\edge_j} \big)\big|_{\edge_j} \cdot \tau_{\edge_j}. 
\end{align*}
This allows us to compute the coefficients $\gamma_j$ similarly as in  Section~\ref{subsec:RedZienkiewicz}, which leads to the additional code displayed in Figure~\ref{fig:MatlabRoutineGNreduced}.

\begin{figure}
\begin{lstlisting}[language=Matlab,escapechar=|]
gam0=(V(10,7:9)-tangents(1,:)/2*(V([2,5],7:9)+V([3,6],7:9)))/V(10,10);
gam1=(V(11,7:9)-tangents(2,:)/2*(V([3,6],7:9)+V([1,4],7:9)))/V(11,11);
gam2=(V(12,7:9)-tangents(3,:)/2*(V([1,4],7:9)+V([2,5],7:9)))/V(12,12);
RedBasis = [eye(9);zeros(3,6),-[gam0;gam1;gam2]];
Basis4elem(:,:,elem) = RedBasis/V(1:9,1:9);
l2g_Dof = [nodes,nrNodes+nodes,2*nrNodes+sides];
\end{lstlisting}
\caption{Modification replacing lines~\ref{line:invVanderGN}--\ref{line:invVanderGN2} in Figure~\ref{fig:MatlabGN} to obtain the reduced \Guzman--Neilan element.}\label{fig:MatlabRoutineGNreduced}
\end{figure}

\section{Numerical experiment}\label{sec:NumExp}

\noindent In this work we have presented an exact numerical integration for rational finite element functions. 
Previously, only inexact quadrature was available. 
In this experiment we investigate the benefits of our exact numerical integration. 

\subsection{Biharmonic eigenvalue problem}\label{subsec:BiHarm}
It is well known that the quadrature error may spoil the approximation order when approximating solutions to PDEs as shown for the linear elliptic PDE $-\divergence( a \nabla u) = f$ for smooth coefficient function $a$ discretized by Lagrange elements, cf.~\cite[Thm.~4.1.6]{Ciarlet2002} and for the related eigenvalue problem~\cite{BanerjeeOsborn90,Banerjee92}. 
The results state that for smooth solutions and Lagrange elements of degree $k$ the order of convergence remains optimal if the quadrature rule is exact for polynomials of maximal degree $2k-2$ for the source problem and of  maximal degree $2k-1$ for the eigenvalue problem. 
For rough coefficients $a$, randomized quadrature rules have been investigated in~\cite{KrusePolydoridesWu19}. 

In the following we numerically study the effect of quadrature errors in the singular \Zienkiewicz FEM for the biharmonic eigenvalue problem: Seek an eigenfunction $\phi\in H^2(\Omega)\setminus \lbrace 0 \rbrace$ with smallest eigenvalue $\lambda > 0$ such that
\begin{equation*}
\begin{aligned}
\Delta^2 \phi &= \lambda \, \phi && \text{in }\Omega,\\
\phi & = \nabla \phi\cdot \nu = 0 &&\text{on }\partial \Omega.
\end{aligned}
\end{equation*} 
The domain $\Omega$ is chosen as the unit square $\Omega = (0,1)^2$ or the L-shaped domain $\Omega = (-1,1)^2\setminus [0,1)^2$. The underlying triangulations are uniform for the square and graded towards the re-entrant corner for the L-shaped domain.
The grading is obtained by the standard AFEM loop with D\"orfler marking and bulk parameter $\theta = 0.5$, where the refinement indicator is for all $T\in \tria$ defined as 
\begin{align*}
\eta^2(T) \coloneqq \lvert \textup{mid}(T)\rvert^{-2} |T|^{5/7}.
\end{align*}%
For the (inexact) quadrature we use Fubini's theorem and 1D Gau\ss{} quadrature, that is,  for a function  $g\colon T \to \mathbb{R}$  with $T = [(0,0)^\top,(1,0)^\top,(0,1)^\top]$ we approximate  the integral with suitable Gau\ss{} points $x_i,y_{i,j}\in T$ and weights $\omega_i,\omega_{i,j}\in \mathbb{R}$ for $i,j=1,\dots,n$ with $n\in \mathbb{N}$ by
\begin{equation}\label{eq:Quad}
\begin{aligned}
\int_T g(x,y)\,\mathrm{d}(x,y) &= \int_{0}^1 \int_{0}^{1-x} g(x,y)  \dy
\dx \approx \sum_{i=1}^n \omega_i \int_{0}^{1-x} g(x_i,y) \dy\\
& \approx \sum_{i=1}^n \omega_i \sum_{j=1}^n \omega_{i,j} g(x_i,y_{i,j}).
\end{aligned}
\end{equation}
Notice that for polynomials $g\in \mathcal{P}_k(T)$ the integrand $\int_0^{1-x} g(x,y) \dy$ is a polynomial in $x$ of maximal degree $k+1$. Thus, exploiting the exactness of  Gau\ss{} quadrature for polynomials with maximal degree $2n-1$ in 1D, this approach is exact for polynomials $g\in \mathcal{P}_{2n-2}(T)$. We use the quadrature in~\eqref{eq:Quad} to approximate the integrands in~\eqref{eq:bilin-T} and the local contributions of the mass matrix 
\begin{equation*}
(M_T)_{r,s} = \int_T b_r b_s \dx \qquad\text{for all }r,s =1,\dots,12.
\end{equation*}
We solve the resulting eigenvalue problem and compare the approximated eigenvalue $\lambda_h$ obtained with exact integration by the implementation presented in Section~\ref{sec:implZienk} and the eigenvalue $\overline{\lambda}_h$ obtained by numerical quadrature~\eqref{eq:Quad} in the computation of the system matrices. The results are displayed in Figure~\ref{fig:ExpSquare}.
\begin{figure}
{\centering
\begin{tikzpicture}
\begin{axis}[
clip=false,
width=.5\textwidth,
height=.45\textwidth,
xmode = log,
ymode = log,
xmax = 6000000,
cycle multi list={\nextlist MyColors},
scale = {1},
xlabel={ndof},
clip = true,
legend cell align=left,
legend style={legend columns=1,legend pos= outer north east,font=\fontsize{7}{5}\selectfont}
]
 \addplot table [x=ndof,y=relatDistToExactQuad] {Experiments/Exp1_Square_n2.txt} node[pos=1.0, right] {\tiny $n=2$};
    \addplot table [x=ndof,y=relatDistToExactQuad] {Experiments/Exp1_Square_n3.txt} node[pos=1.0, right] {\tiny $n=3$};
    \addplot table [x=ndof,y=relatDistToExactQuad] {Experiments/Exp1_Square_n4.txt} node[pos=1.0, right] {\tiny $n=4$};
    \addplot table [x=ndof,y=relatDistToExactQuad] {Experiments/Exp1_Square_n5.txt} node[pos=1.0, right] {\tiny $n=5$};
    \addplot table [x=ndof,y=relatDistToExactQuad] {Experiments/Exp1_Square_n6.txt} node[pos=1.0, right] {\tiny $n=6$};
    \addplot table [x=ndof,y=relatDistToExactQuad] {Experiments/Exp1_Square_n7.txt} node[pos=1.0, right] {\tiny $n=7$};
    \addplot table [x=ndof,y=relatDistToExactQuad] {Experiments/Exp1_Square_n8.txt} node[pos=1.0, right] {\tiny $n=8$};
    \addplot table [x=ndof,y=relatDistToExactQuad] {Experiments/Exp1_Square_n9.txt} node[pos=1.0, right] {\tiny $n=9$};
    \addplot table [x=ndof,y=relatDistToExactQuad] {Experiments/Exp1_Square_n10.txt} node[pos=1.0, right] {\tiny $n=10$};
    \addplot table [x=ndof,y=relatDistToExactQuad] {Experiments/Exp1_Square_n11.txt} node[pos=1.0, right] {\tiny $n=11$};
\end{axis}
\end{tikzpicture}
\begin{tikzpicture}
\begin{axis}[
clip=false,
width=.5\textwidth,
height=.45\textwidth,
xmode = log,
ymode = log,
xmax = 2500000,
cycle multi list={\nextlist MyColors},
scale = {1},
xlabel={ndof},
clip = true,
legend cell align=left,
legend style={legend columns=1,legend pos= south west,font=\fontsize{7}{5}\selectfont}
]
    \addplot[dotted,sharp plot,update limits=false] coordinates {(1e3,2e-2) (1e7,2e-4)};\label{plot:5}
    \addplot[dashed,sharp plot,update limits=false] coordinates {(1e3,1e-5) (1e6,1e-8)};\label{plot:6}
 \addplot table [x=ndof,y=relatDistToExactQuad] {Experiments/Exp2_Lshape_n2} node[pos=1.0, right] {\tiny $n=2$};
    \addplot table [x=ndof,y=relatDistToExactQuad] {Experiments/Exp2_Lshape_n3} node[pos=1.0, right] {\tiny $n=3$};
    \addplot table [x=ndof,y=relatDistToExactQuad] {Experiments/Exp2_Lshape_n4} node[pos=1.0, right] {\tiny $n=4$};
    \addplot table [x=ndof,y=relatDistToExactQuad] {Experiments/Exp2_Lshape_n5} node[pos=1.0, right] {\tiny $n=5$};
    \addplot table [x=ndof,y=relatDistToExactQuad] {Experiments/Exp2_Lshape_n6} node[pos=1.0, right] {\tiny $n=6$};
    \addplot table [x=ndof,y=relatDistToExactQuad] {Experiments/Exp2_Lshape_n7} node[pos=1.0, right] {\tiny $n=7$};
    \addplot table [x=ndof,y=relatDistToExactQuad] {Experiments/Exp2_Lshape_n8} node[pos=1.0, right] {\tiny $n=8$};
    \addplot table [x=ndof,y=relatDistToExactQuad] {Experiments/Exp2_Lshape_n9} node[pos=1.0, right] {\tiny $n=9$};
    \addplot table [x=ndof,y=relatDistToExactQuad] {Experiments/Exp2_Lshape_n10} node[pos=1.0, right] {\tiny $n=10$};
    \addplot table [x=ndof,y=relatDistToExactQuad] {Experiments/Exp2_Lshape_n11} node[pos=1.0, right] {\tiny $n=11$};
    \legend{{$\mathcal{O}(\textup{ndof}^{-1/2})$},{$\mathcal{O}(\textup{ndof}^{-1})$}};	
    \end{axis}
\end{tikzpicture}
}
\caption{Distance $|\lambda_h - \overline{\lambda}_h|/\lambda_h$ on the square with uniform refinement (left) and L-shaped domain with graded mesh (right), where $\overline{\lambda}_h$ resulted from computations with the quadrature rule in~\eqref{eq:Quad} and various $n$.}\label{fig:ExpSquare}
\end{figure}%

Our computations with uniform mesh refinements of the square domain indicate a stagnation of the error as the mesh becomes finer, and hence suggests a failure of convergence. 
This effect is quite strong for small number of Gau\ss{} points $n$, and seems to be less prominent for larger $n$.  
However, it can be observed, that the onset of the stagnation is merely shifted for larger $n$.  

Also for the graded meshes there are some issues. 
While the solutions derived with approximated system matrices seem to exhibit convergence, for small numbers of Gau\ss{} points $n$ one can observe a reduced convergence order $\textup{ndof}^{-1/2}$, instead of $\textup{ndof}^{-1}$ as for larger $n$.   
Notably, the accelerated rate $\textup{ndof}^{-1}$ reflects the convergence of the exact scheme towards the minimal eigenvalue, that is, $\lambda_h - \lambda = \mathcal{O}(\textup{ndof}^{-1})$, cf.~\cite[Sec.~4.3]{CarstensenPuttkammer23}. 
Thus, smaller values of $n$ retard the convergence rate of the numerical scheme, whereas larger  values of $n$ restore the rate to its expected trajectory. 
However, this may well be true only in the pre-asymptotic stage. 
With an increase in the degrees of freedom, the schemes  employing inexact quadrature might still encounter a stagnation, similar to the one  observed for uniform mesh refinement of the unit square domain.

\subsection{Pressure robustness}\label{subsec:ExpPressRob}
The primary advantage of \Guzman{}--Neilan elements over conventional low-order discretizations for fluid problems lies in their exact divergence-free property. Specifically, given the discrete spaces $V$ and $Q$ defined in Lemma~\ref{lem:GN}, we have
\begin{align*}
\left\lbrace v_h \in V\colon \int_\Omega q_h \, \divergence v_h \dx = 0 \text{ for all } q_h \in Q \right\rbrace = \lbrace v_h \in V\colon \divergence v_h = 0 \rbrace.
\end{align*}
However, employing an inexact quadrature compromises this property, as the identity above no longer holds when the integral $\int_\Omega q_h \, \divergence v_h \dx$ is replaced by a numerical approximation. In the remainder of this subsection, we numerically investigate this limitation.
In particular, we solve the following problem from \cite[Sec.~1.1]{John17}. Let $\Omega =(0,1)^2$ be the unit square domain and consider the right-hand side
$f(x,y) \coloneqq (0,100\, (1 - y + 3y^2))^\top$ for all $(x,y)\in \Omega$.
We seek the velocity $u\in H^1_0(\Omega)^2$ and the pressure $p\in L^2_0(\Omega)$ satisfying
\begin{align*}
\begin{aligned}
-\Delta u + \nabla p &= f &&\quad \text{in } \Omega, \\
-\divergence u &= 0 &&\quad \text{in } \Omega.
\end{aligned}
\end{align*}
The exact solution to the Stokes problem is given by $u = 0$ and $p(x,y) = 100\, (y^3 - y^2/2 + y - 7/12)$. 
A pressure-robust numerical scheme, such as the \Guzman{}--Neilan FEM with exact quadrature, reproduces the exact velocity $u_h = u = 0$.
With inexact quadrature such as the one defined in \eqref{eq:Quad} with $n$ 1D Gau\ss{} points this property no longer holds, as illustrated in Figure~\ref{fig:Exp3}. In fact, for a moderate number of Gau\ss{} points, the velocity error in the \Guzman{}--Neilan FEM with inexact quadrature exceeds that of the classical Taylor--Hood FEM \cite{Boffi2013,DieningStornTscherpel22}, which is known to struggle with pressure robustness, cf.~\cite{John17}.
Thus, inexact quadrature undermines the key advantage of the \Guzman{}--Neilan FEM, highlighting the crucial role of exact quadrature in preserving its divergence-free property.
\begin{figure}
\begin{tikzpicture}
\begin{axis}[
clip=false,
width=.5\textwidth,
height=.45\textwidth,
ymode = log,
cycle multi list={\nextlist MyColors},
scale = {1},
xlabel={$n$},
clip = true,
legend cell align=left,
legend style={legend columns=1,legend pos= north east,font=\fontsize{7}{5}\selectfont}
]
 \addplot table [x=nrGpts,y=GradNorm] {Experiments/Exp3_Stokes_red.txt};
  \addplot[dashed,sharp plot,update limits=false] coordinates {(-5,4.410009e-05) (16,4.410009e-05)};\label{line:TH}
    \legend{{$\lVert \nabla (u-u_h) \rVert_{L^2(\Omega)}$}};	
    \end{axis}
\end{tikzpicture}
\caption{Velocity error $\lVert \nabla (u - u_h) \rVert_{L^2(\Omega)}$ for the \Guzman{}--Neilan FEM with uniformly refined triangulation $\tria$ with $\#\tria = 8192$ elements and inexact quadrature in \eqref{eq:Quad} with $n$ 1D Gau\ss{} points.
The dashed line \ref{line:TH} represents the velocity error of the Taylor--Hood FEM.
}\label{fig:Exp3}
\end{figure}
\subsection*{Acknowledgements}
We thank Professor Werner Hoffmann for the elegant proof of Lemmas~\ref{lem:QuadPolynomials} and~\ref{lem:QuadRat-1}. 
The work of the first two authors was supported by the Deutsche Forschungsgemeinschaft (DFG, German Research Foundation) – SFB 1283/2 2021 – 317210226.
 
\printbibliography 
\end{document}